\documentclass[12pt,a4paper]{amsart}
\usepackage[czech,english]{babel}
\usepackage{amssymb,eucal}
\usepackage[all,cmtip]{xy}
\usepackage{pb-diagram,pb-xy}
\usepackage{hyperref}

\calclayout

\newcommand{\+}{\nobreakdash-}
\renewcommand{\:}{\colon}
\renewcommand{\.}{\mskip.5\thinmuskip\relax}

\newcommand{\rarrow}{\longrightarrow}
\newcommand{\larrow}{\longleftarrow}
\newcommand{\ot}{\otimes}

\newcommand{\bu}{{\text{\smaller\smaller$\scriptstyle\bullet$}}}
\newcommand{\lrarrow}{\.\relbar\joinrel\relbar\joinrel\rightarrow\.}

\DeclareMathOperator{\Hom}{Hom}
\DeclareMathOperator{\Ext}{Ext}
\DeclareMathOperator{\coker}{coker}
\DeclareMathOperator{\pd}{pd}
\DeclareMathOperator{\id}{id}

\newcommand{\A}{{\mathsf A}}
\newcommand{\B}{{\mathsf B}}
\newcommand{\C}{{\mathsf C}}
\newcommand{\D}{{\mathsf D}}
\newcommand{\E}{{\mathsf E}}
\newcommand{\F}{{\mathsf F}}

\DeclareMathOperator{\Hot}{\mathsf{Hot}}
\DeclareMathOperator{\Add}{\mathsf{Add}}
\DeclareMathOperator{\Prod}{\mathsf{Prod}}
\DeclareMathOperator{\Sets}{\mathsf{Sets}}

\newcommand{\proj}{{\mathsf{proj}}}
\newcommand{\inj}{{\mathsf{inj}}}
\newcommand{\co}{{\mathsf{co}}}
\newcommand{\ctr}{{\mathsf{ctr}}}

\newcommand{\si}{{\mathsf{si}}}
\newcommand{\mx}{{\mathsf{max}}}
\newcommand{\mn}{{\mathsf{min}}}
\newcommand{\sop}{{\mathsf{op}}}
\newcommand{\rop}{{\mathrm{op}}}

\newcommand{\boT}{{\mathbb T}}
\newcommand{\boZ}{{\mathbb Z}}
\newcommand{\boQ}{{\mathbb Q}}

\newcommand{\cC}{{\mathcal C}}
\newcommand{\cL}{{\mathcal L}}
\newcommand{\cM}{{\mathcal M}}
\newcommand{\bcS}{{\boldsymbol{\mathcal S}}}

\newcommand{\modl}{{\operatorname{\mathsf{--mod}}}}
\newcommand{\comodl}{{\operatorname{\mathsf{--comod}}}}
\newcommand{\contra}{{\operatorname{\mathsf{--contra}}}}
\newcommand{\simodl}{{\operatorname{\mathsf{--simod}}}}
\newcommand{\sicntr}{{\operatorname{\mathsf{--sicntr}}}}
\newcommand{\dproj}{{\mathsf{-proj}}}
\newcommand{\dinj}{{\mathsf{-inj}}}

\theoremstyle{plain}
\newtheorem{thm}{Theorem}[section]

\newtheorem{lem}[thm]{Lemma}
\newtheorem{prop}[thm]{Proposition}
\newtheorem{cor}[thm]{Corollary}
\theoremstyle{definition}
\newtheorem{ex}[thm]{Example}

\newtheorem{rem}[thm]{Remark}

\newcommand{\Section}[1]{\bigskip\section{#1}\medskip}

\begin{document}

\title{$\infty$\+Tilting Theory}

\author{Leonid Positselski and Jan \v S\v tov\'\i\v cek}

\address{Leonid Positselski, Institute of Mathematics, Czech Academy
of Sciences, \v Zitn\'a~25, 115~67 Prague~1, Czech Republic; and
\newline\indent Laboratory of Algebraic Geometry, National Research
University Higher School of Economics, Moscow 119048; and
\newline\indent Sector of Algebra and Number Theory, Institute for
Information Transmission Problems, Moscow 127051, Russia; and
\newline\indent Department of Mathematics, Faculty of Natural Sciences,
University of Haifa, Mount Carmel, Haifa 31905, Israel}

\email{posic@mccme.ru}

\address{Jan {\v S}{\v{t}}ov{\'{\i}}{\v{c}}ek, Charles University
in Prague, Faculty of Mathematics and Physics, Department of Algebra,
Sokolovsk\'a 83, 186 75 Praha, Czech Republic}

\email{stovicek@karlin.mff.cuni.cz}

\begin{abstract}
 We define the notion of an infinitely generated tilting object of
infinite homological dimension in an abelian category.
 A one-to-one correspondence between $\infty$\+tilting objects in
complete, cocomplete abelian categories with an injective cogenerator
and $\infty$\+cotilting objects in complete, cocomplete abelian
categories with a projective generator is constructed.
 We also introduce $\infty$\+tilting pairs, consisting of
an $\infty$\+tilting object and its $\infty$\+tilting class, and obtain
a bijective correspondence between $\infty$\+tilting and
$\infty$\+cotilting pairs.
 Finally, we discuss the related derived equivalences and t\+structures.
\end{abstract}

\maketitle

\tableofcontents

\section*{Introduction}
\medskip

% -- Introduction to our results and their significance

The phrase `tilting theory' is often used to refer to a well-developed general machinery for producing equivalences between triangulated categories (see \cite{Handbook} for an introduction, history and applications). Such equivalences are often represented by a distinguished object, a so-called tilting object, and it is crucial to most of the theory that such a tilting object is homologically small. If $\A$ is an abelian category with exact coproducts (e.g.\ a category of modules over a ring or sheaves on a topological space) and $T$ is a tilting object, the smallness typically translates at least to the assumptions that $T$ is finitely generated and of finite projective dimension.

In this paper we introduce and systematically develop $\infty$-tilting theory, where all homological smallness assumptions are dropped. This brings under one roof various concepts and results from the literature:
\begin{enumerate}
\item Wakamatsu tilting modules~\cite{MR,Wak1,Wak2} over finite dimensional algebras,
\item semidualizing bimodules and the Foxby equivalence~\cite{Ch,HW,Pps}, and
\item the comodule-contramodule~\cite[Section 0.2 and Chap.\ 5]{Psemi} and the semimodule-semicontramodule~\cite[Sections~0.3.7
and~6.3]{Psemi} correspondences.
\end{enumerate}

These results come with rather different motivations: from criteria for stable equivalences of finite dimensional self-injective algebras in~(1), through Gorenstein homological algebra in~(2), to the representation theory of infinite-dimensional Lie algebras (e.g.\ the Virasoro or Kac--Moody algebras) in~(3).

A part of the work has been done in our previous paper~\cite{PS}, where we explained how the finite generation assumption can be naturally dropped with help of additive monads and, in several cases of interest, with topological rings.

In this paper we focus on dropping the assumption of finite homological dimension. It turns out that we still obtain triangulated equivalences and (co)tilting t\+structures, but in general not for the conventional derived categories of two abelian categories, but rather for a so-called \emph{pseudo-coderived} category of one of them and a \emph{pseudo-contraderived} category of the other.

Here, a pseudo-coderived category of an abelian category $\A$ is a certain triangulated category $\D$ to which $\A$ fully embeds as the heart of a t\+structure and such that $\Ext^i_\A(X,Y)$ is canonically isomorphic to $\Hom_\D(X,Y[i])$ for all $X,Y\in\A$ and $i\ge 0$. The term `pseudo-coderived' comes from the fact that, under reasonable assumptions satisfied in particular in the situations~(1--3) above, the pseudo-coderived category is an intermediate Verdier quotient between the conventional derived category $\D(\A)$ and the coderived category $\D^\co(\A)$ (which is none other than the homotopy category $\Hot(\A_\inj)$ of complexes of injective objects if $\A$ is a locally Noetherian Grothendieck category). A pseudo-contraderived category has formally dual properties. Pseudo-co/contraderived categories are in fact not determined uniquely by their abelian hearts, but depend on a certain parameter, so that we often do not get just a single triangulated equivalence, but rather a family of compatible triangulated equivalences. We refer to \cite{Pps} for an in-depth discussion of this new class of triangulated categories.

\medskip

% -- History of tilting theory for finite projective dimension

 To put our results into context, we briefly recall the history 
of tilting theory, which evolved through a series of successive generalizations
in several directions.
 The definition of what is now known as a finitely generated tilting
module of projective dimension~$1$ over a finite-dimensional
associative algebra first appeared in the paper of Happel and
Ringel~\cite{HR} (see also Bongartz~\cite{Bo}), who were building
upon a previous work of Brenner and Butler~\cite{BB}.
 The main result was the so-called \emph{Tilting Theorem}, or
the \emph{Brenner--Butler theorem}, establishing equivalences
between certain additive subcategories of the categories of
finitely-generated modules over an algebra $R$ and over
the endomorphism algebra $S$ of a tilting $R$\+module.
 Happel~\cite{Ha} proved that a tilting module
induces a triangulated equivalence between the derived categories
of finitely-generated $R$\+modules and $S$\+modules.

 Finitely presented tilting modules of projective dimension~$1$
over arbitrary rings were discussed by Colby and Fuller~\cite{CF},
while finitely presented tilting modules of arbitrary finite projective
dimension~$n$ were studied already by Miyashita~\cite{Mi} and
Cline--Parshall--Scott~\cite{CPS}.
 The tilting theorem (for categories of infinitely generated modules)
was proved in~\cite{Mi}, and the related derived equivalence was
constructed in~\cite{CPS}.
 Infinitely generated tilting modules of projective dimension~$1$
(now also known as \emph{big $1$\+tilting modules}) were defined
by Colpi and Trlifaj~\cite{CT}.
 The tilting theorem for self-small tilting objects of projective
dimension~$1$ in Grothendieck abelian categories was obtained
by Colpi~\cite{Co}.
 Cotilting modules of injective dimension~$1$ were introduced
by Colby--Fuller~\cite{CF} and Colpi--D'Este--Tonolo~\cite{CDT}
(see also~\cite{CTT}).
 Finally, infinitely generated tilting modules of projective
dimension~$n$ and cotilting modules of injective dimension~$n$
(\emph{big $n$\+tilting} and \emph{$n$\+cotilting modules})
were defined by Angeleri and Coelho~\cite{AC} and
characterized by Bazzoni~\cite{Ba}.

 The main results of the infinitely generated tilting theory claim
that all $n$\+tilting modules are of finite type~\cite{BS} and
all $n$\+cotilting modules are pure-injective~\cite{St0}.
 The tilting theorem for big $1$\+tilting modules was obtained
in some form by Gregorio and Tonolo~\cite{GT}.
 Another approach, based on a previous work by Facchini, was
developed by Bazzoni~\cite{Ba2}, who also proved that
the derived category of $R$\+modules is equivalent to a full
subcategory and a quotient category of the derived category of
$S$\+modules when $S$ is the endomorphism ring of a big
$1$\+tilting $R$\+module.
 This was extended to big $n$\+tilting modules by
Bazzoni, Mantese, and Tonolo~\cite{BMT}.
 A correspondence between $n$\+cotilting modules and small $n$\+tilting
objects in Grothendieck abelian categories together with the related
derived equivalence were constructed by the second author of the present
paper~\cite{St}.
 Big $n$\+tilting objects in abelian categories were defined and
the related derived equivalence was obtained by
Nicol\'as--Saor\'\i n--Zvonareva~\cite{NSZ} and
Fiorot--Mattiello--Saor\'\i n~\cite{FMS} (see also
Psaroudakis--Vit\'oria~\cite{PV}).
 Finally, a correspondence between big $n$\+tilting and $n$\+cotilting
objects in abelian categories was constructed in the paper~\cite{PS}
by the two present authors.

 The main innovation in~\cite{PS}, which allows to obtain very naturally
derived equivalences from big $n$-tilting objects and which is
based on the ideas previously
developed in~\cite{St}, \cite{PV}, \cite{NSZ}, and~\cite{PR}, is that
to a big tilting object $T$ in an abelian category $\A$ one can assign
a richer structure than its ring of endomorphisms $\Hom_\A(T,T)$.
 For any set $X$, consider the set of all morphisms $T\rarrow T^{(X)}$
in $\A$, where $T^{(X)}$ denotes the coproduct of $X$ copies of~$T$.
 Then the endofunctor $X\longmapsto\Hom_\A(T,T^{(X)})$ is a monad on
the category of sets.
 The tilting heart $\B$ corresponding to the tilting object $T\in\A$
is the abelian category of all algebras (which we also call modules)
over this monad. In many naturally occurring situations, one can equip
$\Hom_\A(T,T)$ with a complete and separated topology so that
$\Hom_\A(T,T^{(X)})$ identifies with families of elements of $\Hom_\A(T,T)$
indexed by $X$ which converge to zero.

% -- Wakamatsu tilting as a smaller brother

 A notion of a finitely generated tilting module of infinite
projective dimension (now known as \emph{Wakamatsu tilting modules}) was
introduced in the representation theory of finite-dimensional algebras 
by Wakamatsu in~\cite{Wak1,Wak2} and it was studied further
by Mantese--Reiten in~\cite{MR}.

\medskip

% -- Structure of the paper

 In the present paper we work out a common generalization
of two lines of thought described above, namely of
big $n$-tilting/cotilting modules
and finite-dimensional Wakamatsu tilting modules.
We develop a theory
of big tilting and cotilting objects of possibly infinite homological
dimension in abelian categories.
 Our goal is also to put on a rigorous footing
the discussion of ``$\infty$\+tilting objects''
%in~\cite[Examples~5.2 and 5.5, and Section~9.2]{PS}.
in~\cite[Sections~10.1, 10.2, and~10.3]{PS}.
The structure of the paper is as follows.

To a complete, cocomplete
abelian category $\A$ with an injective cogenerator $J$ and an $\infty$-tilting
object $T$ we associate in Section~\ref{infty-tilting-cotilting-secn} a complete,
cocomplete abelian category $\B$ with a projective generator $P$ and
an $\infty$-cotilting object $W$. We do so in such a way that, up to equivalence,
this induces a bijective correspondence
between the triples $(\A,T,J)$ and $(\B,P,W)$.

% One of the main innovations in this paper is the observation that
%the tilting derived equivalences in the Wakamatsu tilting theory
%do not, in general, connect the conventional derived categories of
%the two abelian categories, but rather some of their
%\emph{pseudo-coderived} and \emph{pseudo-contraderived} categories.
% The latter terms mean triangulated categories intermediate between
%the conventional and the co/contraderived categories.

 In order to obtain the announced version of derived equivalences, we need
to associate to each $\infty$\+tilting object a certain coresolving subcategory
$\E\subset\A$ which plays the role of the tilting class in~\cite{PS}.
This is discussed in Section~\ref{infty-til-cotil-pairs-secn}.
Such a class $\E$ is in general not unique, but the possible choices form
a complete lattice with respect to the inclusion.
% We call the pair $(T,\.\E)$ an $\infty$\+tilting pair in $\A$.
 Having chosen $\E$, we already
obtain a uniquely determined full subcategory $\F\subset\B$
which plays the role of a cotilting class, and equivalences
$\E\simeq\F$ and $\D(\E)\simeq\D(\F)$.
Each of $\D(\E)$ and $\D(\F)$ comes naturally equipped with two t\+structures,
and the two abelian categories $\A$ and $\B$ are the hearts of these
two t\+structures (see Section~\ref{t-structures-secn}).

 If, moreover, the $\infty$\+tilting class $\E$ is closed under coproducts in $\A$
and the $\infty$\+cotilting class $\F$ is closed under products in $\B$,
then we show in Section~\ref{derived-equivalence-secn}
that $\D(\E)$ is a pseudo-coderived category and $\D(\F)$ is
a pseudo-contraderived category in the sense of~\cite{Pps}.
 Although the above closure properties of $\E$ and $\F$ are not automatic
in our setup, they
are satisfied for our motivating classes of examples mentioned above and,
in details, also in Section~\ref{examples-secn}.

\medskip

\textbf{Acknowledgement.}
 The authors are grateful to Jan Trlifaj, Sefi Ladkani, and Luisa Fiorot
for helpful discussions and comments.
 Leonid Positselski's research is supported by research plan
RVO:~67985840, by the Israel Science Foundation grant~\#\,446/15, and
by the Grant Agency of the Czech Republic under the grant P201/12/G028.
Jan \v{S}\v{t}ov\'{\i}\v{c}ek's research is supported by
the Grant Agency of the Czech Republic under the grant P201/12/G028.

\Section{The Tilted and Cotilted Abelian Categories}
\label{co-tilted-categories}

 Given an additive category $\C$ with set-indexed coproducts and
an object $M\in\C$, we denote by $\Add(M)\subset\C$ the full
subcategory formed by the direct summands of coproducts of copies
of $M$ in~$\C$.
 Similarly, given an additive category $\C$ with set-indexed
products and an object $L\in\C$, we denote by $\Prod(L)\subset\C$
the full subcategory formed by the direct summands of products of
copies of $L$ in~$\C$.
 Given a set $X$, the coproduct of $X$ copies of $M$ is denoted
by $M^{(X)}\in\Add(M)$ and the product of $X$ copies of $L$ is
denoted by $L^X\in\Prod(L)$.

 We say that an additive category is \emph{idempotent-complete}
(or in other terminology~Karoubian or pseudo-abelian) if it contains the images of
all idempotent endomorphisms of its objects.

\begin{thm} \label{add-prod-theorem}
\textup{(a)} Let\/ $\C$ be an idempotent-complete additive category
with coproducts and $M\in\C$ be an object.
 Then there exists a unique abelian category\/ $\B$ with enough
projective objects such that the full subcategory of projective
objects\/ $\B_\proj\subset\B$ is equivalent to the full subcategory\/
$\Add(M)\subset\C$.
 The abelian category\/ $\B$ has products, coproducts, and a natural
projective generator $P\in\B_\proj$ corresponding to
the object $M\in\Add(M)$. \par
\textup{(b)} Let\/ $\C$ be an idempotent-complete additive category
with products and $L\in\C$ be an object.
 Then there exists a unique abelian category\/ $\A$ with enough
injective objects such that the full subcategory of injective
objects\/ $\A_\inj\subset\A$ is equivalent to the full subcategory\/
$\Prod(L)\subset\C$.
 The abelian category\/ $\A$ has products, coproducts, and a natural
injective cogenerator $J\in\A_\inj$ corresponding to the object
$L\in\Prod(L)$.
\end{thm}

\begin{proof}
 Part~(a): the category $\B$ is unique, because an abelian category
with enough projective objects is determined by its full subcategory
of projective objects~\cite[proof of Theorem~6.2]{St},
\cite[proof of Theorem~3.6]{Prev}.

 To prove existence, one can construct $\B$ as the category of
finitely presented (coherent) contravariant functors on $\Add(M)$.
 This category can be also described as the quotient category of
the category $\Add(M)^2$ of morphisms in $\Add(M)$ by the ideal
of all morphisms in $\Add(M)^2$ which factorize through objects
of the full subcategory in $\Add(M)^2$ consisting of all the split
epimorphisms in $\Add(M)$ \cite{Bel}.
 The category $\B$ is abelian, because the additive category $\Add(M)$
is right coherent (has weak kernels) \cite[Corollary~1.5]{Fr},
\cite[Lemma~2.2 and Proposition~2.3]{Kra0},
\cite[Proposition~4.5(1)]{Bel}, \cite[Lemma~1(1)]{Kra} (see
also~\cite[Appendix~B]{Fio} for a further discussion and references).
 Indeed, if $f\:M'\rarrow M''$ is a morphism in $\Add(M)$ and
$X$ is the set of all morphisms $M\rarrow M'$ whose composition
with~$f$ vanishes, then the natural morphism $M^{(X)}\rarrow
M'$ is a weak kernel of~$f$ in $\Add(M)$ (cf.~\cite[Lemma~2(1)]{Kra}).

 Even more explicitly, $\B$ is the category of modules over
the monad $\boT\:X\longmapsto\Hom_\C(M,M^{(X)})$ on the category
of sets (we call modules here what is often called monadic $\boT$-algebras,
since they generalize ordinary modules over a ring; see the discussions
in the introduction to~\cite{PR}, \cite[Lemma~1.1 and Example~1.2(2)]{Pper},
and~\cite[Sections~6.1 and~6.3]{PS}, and the references therein).

 Coproducts in the category of coherent functors exist
by~\cite[Lemma~1(2)]{Kra}; more generally, whenever the category
of projective objects $\B_\proj$ in an abelian category $\B$ with
enough projective objects has coproducts, the coproducts in $\B$
can be constructed in terms of the coproducts in $\B_\proj$ (and
the embedding functor $\B_\proj\rarrow\B$ preserves coproducts).
 Products exist in the category of algebras over every monad
$\boT\:\Sets\rarrow\Sets$ and are preserved by the forgetful
functor from the category of $\boT$\+algebras to $\Sets$ (coproducts
also exist in the category of $\boT$\+algebras, but are not preserved
by the forgetful functor).
 The natural projective generator $P\in\B_\proj$ is the free
$\boT$\+algebra/module with one generator.

 Part~(b) is dual to~(a).  Explicitly, $\A$ is the opposite category to
the category of coherent covariant functors on $\Prod(L)$, or
the opposite category to the category of modules over the monad
$\boT\:X\longmapsto\Hom_\C(L^X,L)$ on the category of sets.
\end{proof}

 We will use the notation $\B=\sigma_M(\C)$ and $\A=\pi_L(\C)$.
 Assuming that $M\in\C$ is a ``tilting object'' in one sense or
another (cf.\ the next Section~\ref{infty-tilting-cotilting-secn}),
one can call $\B$ the \emph{abelian category tilted from\/ $\C$ at~$M$}.
 Similarly, assuming that $L\in\C$ is a ``cotilting object'' in
some sense, one can call $\A$ the \emph{abelian category cotilted
from\/ $\C$ at~$L$}.

 Now let us assume that $\C$ is an abelian category.
 Then, in the context of Theorem~\ref{add-prod-theorem}(a), the additive
embedding functor $\Phi_\proj\:\B_\proj\simeq\Add(M)\rarrow\C$ can be
uniquely extended to a right exact functor $\Phi\:\B\rarrow\C$.
 To compute the object $\Phi(B)\in\C$ for a given object $B\in\B$,
one can present $B$ as the cokernel of a morphism of projective objects
$f\:P''\rarrow P'$ in $\B$ and put $\Phi(B)=\coker\Phi_\proj(f)$.

 The additive embedding functor $\Add(M)\simeq\B_\proj\rarrow\B$
can be extended to a (left exact) functor $\Psi\:\C\rarrow\B$
right adjoint to~$\Phi$.
% Indeed, since $\Phi$ has a right adjoint $\Psi$
%by \cite[Corollary V.3.2]{Mit}, the latter claim is equivalent to saying
%that $\Psi(\Add(M))\subset\B_\proj$, as then
%$\Psi|_{\Add(M)}\:\Add(M)\rarrow\B_\proj$ becomes a right adjoint to the
%equivalence $\Phi|_{\B_\proj}\:\B_\proj\simeq\Add(M)$. To that end,
%given $M'\in\Add(M)$ and $P'\in\B_\proj$ such that $\Phi(P')\simeq M'$,
%it suffices to observe that $\Hom_\B(-,P')\simeq\Hom_\B(-,\Psi(M'))$ as
%functors from $\B^\rop$ to the category of abelian groups. Since both the
%functors are left exact, it suffices to identify they restrictions to
%$\B_\proj^\rop$, which is achieved by the isomorphisms
%\[
%\Hom_\B(-,P')|_{\B_\proj^\rop} \simeq
%\Hom_\A(\Phi(-),M')|_{\Add(M)} \simeq
%\Hom_\B(-,\Psi(M'))|_{\B_\proj^\rop},
%\]
%induced by the equivalence $\B_\proj\simeq\Add(M)$ and
%the adjunction $(\Phi,\Psi)$, respectively.
%
 Representing the objects of $\B$ as modules over the monad
$\boT\:X\longmapsto\Hom_\C(M,M^{(X)})$ on the category of sets,
one can compute the functor $\Psi$
%more explicitly
as the functor $N\longmapsto\Hom_\C(M,N)$, with the $\boT$\+module
structure on the set $\Hom_\C(M,N)$ constructed as explained
in~\cite[Section~6.3]{PS} (in this case $\Hom_\C(M,N)$ of course carries the structure of an abelian group, even a right $\boT(*)$-module, where $*$ stands for a one-element set).

 Indeed, let us show that the functor $\Phi$ is left adjoint to~$\Psi$.
 First of all, the natural projective generator $P\in\B$ (corresponding
to the object $M\in\Add(M)$) corepresents the forgetful functor
from the category $\B\simeq\boT\modl$ to the category of sets or abelian
groups, that is, for any object $B\in\boT\modl$ one has $\Hom_\B(P,B)
\simeq B$.
 In particular, for any object $N\in\C$ we have a natural isomorphism
of the Hom groups
$$
 \Hom_\B(P,\Psi(N))=\Hom_\B(P,\Hom_\C(M,N))\simeq\Hom_\C(M,N)=
 \Hom_\C(\Phi(P),N).
$$
 Hence for any set $X$ there are natural isomorphisms
$$
 \Hom_\B(P^{(X)},\Psi(N))\simeq\Hom_\C(M,N)^{X}\simeq
 \Hom_\C(M^{(X)},N)\simeq\Hom_\C(\Phi(P^{(X)}),N).
$$
 Passing to the direct summands, we obtain a natural isomorphism
of the Hom groups
$$
 \Hom_\B(P',\Psi(N))\simeq\Hom_\C(\Phi(P'),N)
$$
for all objects $P'\in\B_\proj$ and $N\in\C$.
 This isomorphism is clearly functorial in an object $N\in\C$; and
the construction of the action of the monad $\boT$ on the set
$\Psi(N)=\Hom_\C(M,N)$ in~\cite[proof of Proposition~6.2 and
Remark~6.4]{PS} is designed
so as to make these isomorphisms compatible with all the morphisms
$P''\rarrow P'$ in the category $\B_\proj\simeq\Add(M)$.
 Finally, both the contravariant functors $\Hom_\B({-},\Psi(N))$ and
$\Hom_\C(\Phi({-}),N)$ take the cokernels of morphisms in $\B$ to
the kernels of morphisms of abelian groups, so our isomorphism of
the Hom groups extends from $P'\in\B_\proj$ to all objects $B\in\B$.

 Similarly, in the context of Theorem~\ref{add-prod-theorem}(b),
the additive embedding functor $\Psi_\inj\:\A_\inj\simeq\Prod(L)
\rarrow\C$ can be uniquely extended to a left exact functor
$\Psi\:\A\rarrow\C$.
 The additive embedding functor $\Prod(L)\simeq\A_\inj\rarrow\A$
can be extended to a (right exact) functor $\Phi\:\C\rarrow\A$
left adjoint to~$\Psi$.

 For more explicit descriptions of abelian categories $\B$ arising
in connection with objects $M$ in more specific classes of additive
categories $\C$ in Theorem~\ref{add-prod-theorem}(a), we refer
to~\cite[Theorems~7.1, 9.9, and~9.11, and Proposition~9.1]{PS}.

 The following question will be addressed in the next
Section~\ref{infty-tilting-cotilting-secn}:
given an abelian category $\A$ with coproducts and an object
$M\in\A$, under which assumptions there is an object $L$ in
the abelian category $\B=\sigma_M(\A)$ such that $\pi_L(\B)=\A$\,?
 Similarly, given an abelian category $\B$ with products and an object
$L\in\B$, under which assumptions there is an object $M$ in
the abelian category $\A=\pi_L(\B)$ such that $\sigma_M(\A)=\B$\,?

\Section{\texorpdfstring{$\infty$}{Infinity}-Tilting-Cotilting Correspondence}
%\Section{$\infty$-Tilting-Cotilting Correspondence}
\label{infty-tilting-cotilting-secn}

 Let $\A$ be an abelian category with coproducts.
 We will say that an object $T\in\A$ is \emph{weakly tilting} if one has
$$
 \Ext_\A^i(T,T^{(X)})=0 \quad\text{for all sets~$X$
 and all integers $i>0$}.
$$

 Given two objects $T'\in\Add(T)\subset\A$ and $A\in\A$, a morphism
$t\:T'\rarrow A$ is said to be an \emph{$\Add(T)$\+precover} if
every morphism $t''\:T''\rarrow A$ with $T''\in\Add(T)$ factorizes
through the morphism~$t$.
 Equivalently, this means that the map of abelian groups
$\Hom_\A(T,t)\:\Hom_\A(T,T')\rarrow\Hom_\A(T,A)$ is surjective.
 For every object $A\in\A$, the natural morphism
$T^{(\Hom_\A(T,A))}\rarrow A$ is an $\Add(T)$\+precover.

 Let $T\in\A$ be a weakly tilting object.
 By the definition, the full subcategory $\E_\mx(T)\subset\A$ consists
of all the objects $E\in\A$ satisfying the following two conditions:
\begin{enumerate}
\renewcommand{\theenumi}{\roman{enumi}$_\mx$}
\item $\Ext^i_\A(T,E)=0$ for all $i>0$; and
\item there exists an exact sequence
$$
 \dotsb\lrarrow T_2\lrarrow T_1\lrarrow T_0\lrarrow E\lrarrow0
$$
in $\A$ such that $T_j\in\Add(T)$ for all $j\ge0$ and the sequence
remains exact after applying the functor $\Hom_\A(T,{-})$.
\end{enumerate}

 Notice that the condition of exactness of the sequence of
abelian groups obtained by applying $\Hom_\A(T,{-})$ in~(ii$_\mx$)
can be equivalently restated as the condition that the images $Z_j$
of the morphisms $T_{j+1}\rarrow T_j$ satisfy $\Ext_\A^1(T,Z_j)=0$
for all $j\ge0$.
 In this case, assuming~(i$_\mx$), one also has $\Ext_\A^i(T,Z_j)=0$
for all $j\ge0$ and $i>0$.
 As (ii$_\mx$)~is obviously satisfied for $Z_j$, it follows that
$Z_j\in\E_\mx(T)$ for all $j\ge0$.

 Conversely, given a short exact sequence $0\rarrow Z_0\rarrow T_0
\rarrow E\rarrow0$ with $E$ satisfying~(i$_\mx$), $Z_0$
satisfying~(ii$_\mx$), $T_0\in\Add(T)$, and $\Hom_\A(T,T_0)\rarrow
\Hom_\A(T,E)$ a surjective map, one clearly has $E\in\E_\mx(T)$.

 The following lemma is a generalization
of~\cite[Proposition~2.6]{Wak2}.

\begin{lem} \label{wakamatsu-lemma}
 For any weakly tilting object $T\in\A$, the full subcategory\/
$\E_\mx(T)$ in the abelian category\/ $\A$ is closed under \par
\textup{(a)} extensions, \par
\textup{(b)} the cokernels of monomorphisms, \par
\textup{(c)} the kernels of those epimorphisms which remain
epimorphisms after applying the functor\/ $\Hom_\A(T,{-})$, and \par
\textup{(d)} direct summands.
\end{lem}

\begin{proof}
 To prove parts~(a\+c), consider a short exact sequence
$0\rarrow E'\rarrow E\rarrow E''\rarrow0$ in the abelian category~$\A$.
 Part~(a): clearly, the object $E$ satisfies the condition~(i$_\mx$)
whenever the objects $E'$ and $E''$ do.
 Suppose that $T_0'\rarrow E'$ and $T_0''\rarrow E''$ are epimorphisms
onto the objects $E'$ and $E''$ from objects $T_0'$, $T_0''\in\Add(T)$
that remain epimorphisms after applying the functor $\Hom_\A(T,{-})$.
 Since $\Ext^1_\A(T_0'',E')=0$, the morphism $T_0''\rarrow E''$
can be lifted to a morphism $T_0''\rarrow E$.
 Hence we obtain a morphism from the split short exact sequence
$0\rarrow T_0'\rarrow T_0'\oplus T_0''\rarrow T_0''\rarrow0$ to
the short exact sequence $0\rarrow E'\rarrow E\rarrow E''\rarrow0$.
 Being an epimorphism at the leftmost and rightmost terms, this
morphism of short exact sequences is also an epimorphism at
the middle term.
 The short sequence of kernels $0\rarrow Z'_0\rarrow Z_0\rarrow
Z''_0\rarrow0$ is then also exact, and the vanishing of
$\Ext_\A^i(T,Z'_0)$ and $\Ext_\A^i(T,Z''_0)$ implies the same of
$\Ext_\A^i(T,Z_0)$. We can thus proceed with the construction of a resolution as in~(ii$_\mx$) inductively.

 Part~(b): clearly, the object $E''$ satisfies the condition~(i$_\mx$)
whenever the objects $E'$ and $E$ do.
 Moreover, the epimorphism $E\rarrow E''$ remains an epimorphism
after applying $\Hom_\A(T,{-})$, since $\Ext_\A^1(T,E')=0$.
 Let $T_0\rarrow E$ be an epimorphism onto $E$ from an object
$T_0\in\Add(T)$ that remains an epimorphism after applying
$\Hom_\A(T,{-})$.
 Then the composition $T_0\rarrow E\rarrow E''$ has the same property.
 Let $Z_0$ and $Z''_0$ be the kernels of the epimorphisms
$T_0\rarrow E$ and $T_0\rarrow E''$.
 Then there is a short exact sequence $0\rarrow Z_0\rarrow Z''_0
\rarrow E'\rarrow0$.
 Assuming that $Z_0\in\E_\mx(T)$, one can apply part~(a) in order to
conclude that $Z''_0\in\E_\mx(T)$, hence $E''\in\E_\mx(T)$.

 Part~(c): let us first show that the kernel of every $\Add(T)$-precover
$t'\colon T'\rarrow\nobreak E$ belongs to $\E_\mx(T)$ whenever
$E\in\E_\mx(T)$.
 By the definition, there exists an $\Add(T)$\+precover $t_0\:
T_0\rarrow E$ with the kernel $Z_0$ belonging to $\E_\mx(T)$.
Consider the following pullback diagram.
\[
\xymatrix{
& Z' \ar@{=}[r] \ar@{>->}[d] & Z' \ar@{>->}[d] \\
Z_0 \ar@{>->}[r] \ar@{=}[d] & S \ar@{->>}[r] \ar@{->>}[d] & T' \ar@{->>}[d]^-{t'} \\
Z_0 \ar@{>->}[r] & T_0 \ar@{->>}[r]_-{t_0} & E
}
\]
As $Z_0$ and $T'\in\E_\mx(T)$, we have $S\in\E_\mx(T)$ by part~(a).
Furthermore, since $t'$ stays an epimorphism after applying $\Hom_\A(T,-)$
and $T'$, $E$ satisfy~(i$_\mx$), it follows that $Z'$ satisfies~(i$_\mx$)
and the middle column splits. Hence there exists a short exact sequence
$0\rarrow T_0\rarrow S\rarrow Z'\rarrow0$ and $Z'\in\E_\mx(T)$ by part~(b).
%This proves the claim. 

 Now we can return to our short exact sequence $0\rarrow E'\rarrow
E\rarrow E''\rarrow\nobreak0$.
 Clearly, if the objects $E$ and $E''$ satisfy~(i$_\mx$) and
the map $\Hom_\A(T,E)\rarrow\Hom_\A(T,E'')$ is surjective,
then the object $E'$ also satisfies~(i$_\mx$).
 Furthermore, if $T_0\rarrow E$ is an $\Add(T)$-precover with
the kernel $Z_0$ and if $Z''_0$ is the kernel of the composition
$T_0\rarrow E\rarrow E''$, then $Z_0$, $Z''_0\in\E_\mx(T)$
by the previous paragraph.
 It remains to apply part~(b) to the short exact sequence $0\rarrow
Z_0\rarrow Z_0''\rarrow E'\rarrow0$ in order to conclude that
$E'\in\E_\mx(T)$.

 Part~(d): Let $E'$ and $E''$ be two objects in $\A$ for which
$E=E'\oplus E''\in\E_\mx(T)$.
% Then both $E'$ and $E''$ belong to $\E_\mx(T)$ by part~(b).
 Then it is obvious that $E'$ and $E''$ satisfy~(i$_\mx$).
% It may be, however, instructive to look more in detail at
% the corresponding resolutions for $E',E''$ as in~(ii$_\mx$).
 Starting from the exact sequence~(ii$_\mx$) for the object $E$,
we will simultaneously construct similar exact sequences for
the two objects $E'$ and~$E''$.
 Applying the construction of part~(b) to the short exact sequence
$0\rarrow E'\rarrow E\rarrow E''\rarrow0$, we get an epic
$\Add(T)$\+precover $T_0\rarrow E''$ with the kernel $Z_0''$ included
into a short exact sequence $0\rarrow Z_0\rarrow Z_0''\rarrow E'
\rarrow0$.
 Applying the same construction to the short exact sequence
$0\rarrow E''\rarrow E\rarrow E'\rarrow 0$, we have an epic
$\Add(T)$\+precover $T_0\rarrow E'$ with the kernel $Z_0'$ included
into a short exact sequence $0\rarrow Z_0\rarrow Z_0'\rarrow E''
\rarrow\nobreak0$.

 Continuing with an epic $\Add(T)$\+precover $T_1\rarrow Z_0$
and applying the construction of part~(a), we obtain an epic
$\Add(T)$\+precover $T_1\oplus T_0\rarrow Z_0''$ with the kernel
$Z_1''$ included into a short exact sequence $0\rarrow Z_1\rarrow
Z_1''\rarrow Z_0'\rarrow0$.
 Proceeding in this way, we obtain an epic $\Add(T)$\+precover
$T_2\oplus T_1\oplus T_0\rarrow Z_1''$ with the kernel $Z_2''$
included into a short exact sequence $0\rarrow Z_2\rarrow Z_2''
\rarrow Z_1'\rarrow0$, an epic $\Add(T)$\+precover
$T_3\oplus T_2\oplus T_1\oplus T_0\rarrow Z_2''$, etc.
 Hence we obtain a long exact sequence satisfying the requirements
of~(ii$_\mx$) for $E''$ of the form
$$
 \dotsb\lrarrow T_2 \oplus T_1\oplus T_0\lrarrow T_0\oplus T_1\lrarrow
 T_0\lrarrow E''\lrarrow0,
$$
and there is a similar sequence of the same form for~$E'$.
\end{proof}

 It follows from Lemma~\ref{wakamatsu-lemma}(c) that, given an object
$E\in\E_\mx(T)$, one can construct an exact sequence~(ii$_\mx$) for it
by choosing an arbitrary $\Add(T)$\+precover $T_0\rarrow E$, taking
its kernel $Z_0$, choosing an arbitrary $\Add(T)$\+precover
$T_1\rarrow Z_0$, etc.
 Whichever $\Add(T)$\+precovers one chooses, all the subsequent
$\Add(T)$\+precovers will be epimorphisms, so one will not encounter
any problems in this process.

 In view of Lemma~\ref{wakamatsu-lemma}(a), for any weakly tilting
object $T\in\A$, the full subcategory $\E_\mx(T)\subset\A$ inherits
a Quillen exact category structure from the abelian category~$\A$.
 There are enough projective objects in the exact category $\E_\mx(T)$,
and the full subcategory of projective objects in $\E_\mx(T)$ coincides
with $\Add(T)\subset\E_\mx(T)\subset\A$.

\medskip

Given a full subcategory $\E$ of an idempotent complete
exact category $\A$, we will call $\E$ a \emph{coresolving}
subcategory provided that
\begin{enumerate}
\renewcommand{\theenumi}{\alph{enumi}}
\item $\E$ is closed under extensions, cokernels of admissible
monomorphisms, and direct summands in $\A$, and 
\item $\E$ is cogenerating in $\A$, i.e.\ each $A\in\A$ admits an admissible monomorphism $A\rarrow E$ in $\A$ with $E\in\E$.
\end{enumerate}
Coresolving subcategories provide a suitable
framework to speak of coresolution dimensions of objects
\cite[\S2]{St}, \cite[Ch.~3]{AB69}.

Let now $\A$ be an abelian category with set-indexed
products and an injective cogenerator $J\in\A$.
Then set-indexed coproducts exist and are exact in~$\A$
\cite[Section~2]{PS}.
The full subcategory of injective objects in $\A$
can be described as $\A_\inj=\Prod(J)$.

We will say that an object $T\in\A$ is \emph{$\infty$\+tilting}
(or \emph{big Wakamatsu tilting}) if $T$ is weakly tilting and
$\A_\inj\subset\E_\mx(T)$.
 In this case, the full subcategory $\E_\mx(T)\subset\A$
is coresolving, there are enough injective objects
in the exact category $\E_\mx(T)$, and these are precisely
the injective objects of the ambient abelian category~$\A$.

\medskip

 Now let us present the dual definitions.
 Let $\B$ be an abelian category with products.
 We will say that an object $W\in\B$ is \emph{weakly cotilting}
if one has
$$
 \Ext^i_\B(W^X,W)=0 \quad\text{for all sets $X$
 and all integers $i>0$}.
$$

 Let $W\in\B$ be a weakly cotilting object.
 By the definition, the full subcategory $\F_\mx(W)\subset\B$ consists
of all the objects $F\in\B$ satisfying the two conditions
\begin{enumerate}
\renewcommand{\theenumi}{\roman{enumi}$_\mx^{\textstyle*}$}
\item $\Ext^i_\B(F,W)=0$ for all $i>0$; and
\item there exists an exact sequence
$$
 0\lrarrow F\lrarrow W^0\lrarrow W^1\lrarrow W^2\lrarrow\dotsb 
$$
in $\B$ such that $W^j\in\Prod(W)$ for all $j\ge0$ and the sequence
remains exact after applying the contravariant functor
$\Hom_\B({-},W)$.
\end{enumerate}

\begin{lem} \label{wakamatsu-colemma}
 For any weakly cotilting object $W\in\B$, the full subcategory\/
$\F_\mx(T)$ in the abelian category\/ $\B$ is closed under \par
\textup{(a)} extensions, \par
\textup{(b)} the kernels of epimorphisms, \par
\textup{(c)} the cokernels of those monomorphisms which are
transformed into surjective maps by the contravariant functor\/
$\Hom_\B({-},W)$, and \par
\textup{(d)} direct summands.
\end{lem}

\begin{proof}
 Dual to Lemma~\ref{wakamatsu-lemma}.
\end{proof}

 The definition of a \emph{$\Prod(W)$\+preenvelope} in $\B$ is dual
to the above definition of an $\Add(T)$\+precover in~$\A$.
 The morphism $F\rarrow W^0$ in an exact
sequence~(ii$_\mx^{\textstyle*}$) is a $\Prod(W)$\+preenvelope.
 Denoting the cokernel of this morphism by $Z^0$, the morphism
$Z^0\rarrow W^1$ is also a $\Prod(W)$\+preenvelope, etc.

 Conversely, it follows from Lemma~\ref{wakamatsu-colemma}(c) that,
given any object $F\in\F_\mx(W)$, one can construct an exact
sequence~(ii$_\mx^{\textstyle*}$) for it by choosing an arbitrary
$\Prod(W)$\+preenvelope $F\rarrow W^0$, taking its cokernel $Z^0$,
choosing an arbitrary $\Prod(W)$\+preenvelope $Z^0\rarrow W^1$, etc.
 Whichever $\Prod(W)$\+preenvelopes one chooses in this process, all
the subsequent $\Prod(W)$\+preenvelopes will be monomorphisms, so
one will not encounter any problems.

 In view of Lemma~\ref{wakamatsu-colemma}(a), for any weakly
cotilting object $W\in\B$, the full subcategory $\F_\mx(W)\subset\B$
inherits an exact category structure from the abelian category~$\B$.
 There are enough injective objects in the exact category $\F_\mx(W)$,
and the full subcategory of injective objects in $\F_\mx(W)$ coincides
with $\Prod(W)$.

\medskip

 Let $\B$ be an abelian category with set-indexed coproducts and
a projective generator $P\in\B$.
 Then set-indexed products exist and are exact in~$\B$.
 The full subcategory of projective objects in $\B$ can be described
as $\B_\proj=\Add(P)$.

 We will say that an object $W\in\B$ is \emph{$\infty$\+cotilting}
(or \emph{big Wakamatsu cotilting}) if $W$ is weakly cotilting
and $\B_\proj\subset\F_\mx(T)$.

 When the object $W$ is $\infty$\+cotilting, the full subcategory
$\F_\mx(W)\subset\B$ is \emph{resolving} (i.e.\ generating and closed under extensions, kernels of epimorphisms and direct summands).
 In this case, there are enough projective objects in the exact
category $\F_\mx(W)$, and these are precisely the projective objects
of the ambient abelian category~$\B$.

\begin{thm} \label{tilting-cotilting}
 Let\/ $\A$ be a complete, cocomplete abelian category with
an injective cogenerator $J$ and an\/ $\infty$\+tilting object
$T\in\A$.
 Put\/ $\B=\sigma_T(\A)$, and let\/ $\Phi\:\B\rarrow\A$ be the right
exact functor identifying the full subcategory of projective
objects\/ $\B_\proj\subset\B$ with the full subcategory\/
$\Add(T)\subset\A$.
 Let\/ $\Psi\:\A\rarrow\B$ be the left exact functor right adjoint
to\/ $\Phi$; so $P=\Psi(T)$ is a projective generator of\/~$\B$.
 Set $W=\Psi(J)\in\B$.

 Then $W$ is an\/ $\infty$\+cotilting object in\/~$\B$,
and the restrictions of the functors\/ $\Psi$ and\/ $\Phi$
induce a pair of inverse equivalences of exact categories between
$\E_\mx(T)$ and\/ $\F_\mx(W)$ (see Figure~\ref{tilting-cotilting-fig}),
which identify the\/ $\infty$\+tilting object $T\in\A$ with
the projective generator $P\in\B$ and the\/ $\infty$\+cotilting object
$W\in\B$ with the injective cogenerator $J\in\A$.
%
%the full subcategories\/ $\E_\mx(T)$ and\/ $\F_\mx(W)$ are exact
%and fully faithful, and take them into each other.
 %The two adjoint functors\/ $\Psi|_{\E_\mx(T)}\:\E_\mx(T)\rarrow
%\F_\mx(W)$ and\/ $\Phi|_{\F_\mx(W)}\:\F_\mx(W)\rarrow\E_\mx(T)$
%are mutually inverse equivalences of exact categories taking
%the\/ $\infty$\+tilting object $T\in\A$ to the projective generator
%$P\in\B$ and the\/ $\infty$\+cotilting object $W\in\B$ to
%the injective cogenerator $J\in\A$.
\end{thm}

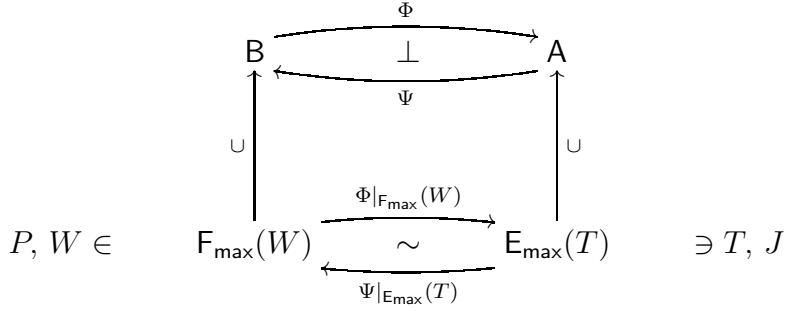
\begin{figure}
\[
\xymatrix{
& \B \ar@/^1ex/@<1ex>[rr]^-{\Phi} & \perp & \A \ar@/^1ex/@<1ex>[ll]^-{\Psi} & \\ \\
P,\, W \in & \F_\mx(W) \ar@/^1ex/@<1ex>[rr]^-{\Phi|_{\F_\mx}(W)} \ar[uu]^-{\cup} & \sim & \E_\mx(T) \ar@/^1ex/@<1ex>[ll]^-{\Psi|_{\E_\mx}(T)} \ar[uu]_-{\cup} & \ni T,\, J \\
}
\]
\caption{Illustration of the $\infty$-Tilting-Cotilting Correspondence
(see Theorems~\ref{tilting-cotilting} and~\ref{cotilting-tilting} and
Corollary~\ref{tilting-cotilting-cor}).}%
\label{tilting-cotilting-fig}%
\end{figure}

\begin{proof}
The functor $\Psi|_{\E_\mx}(T)\:\E_\mx(T)\rarrow\B$ is exact, because
the functor $\Psi$ can be computed as $\Hom_\A(T,{-})$, and
the condition~(i$_\mx$) is imposed.

 To check that the functor $\Psi|_{\E_\mx(T)}$ is fully faithful, one
can choose for any two objects $E'$ and $E''\in\E_\mx(T)$ two initial
fragments $T'_1\rarrow T'_0\rarrow E'\rarrow0$ and $T''_1\rarrow T''_0
\rarrow E''\rarrow0$ of exact sequences~(ii$_\mx$).
 The two sequences being exact in the exact category $\E_\mx(T)$ and
the objects of $\Add(T)$ being projective in $\E_\mx(T)$, one can
compute the group $\Hom_\A(E',E'')$ as the group of all morphisms
$T'_0\rarrow T''_0$ forming a commutative square with some morphism
$T'_1\rarrow T''_1$, modulo those morphisms that come from some
morphism $T'_0\rarrow T''_1$.
 The functor $\Psi$ takes the exact sequences $T_1^{(k)}\rarrow
T_0^{(k)}\rarrow E^{(k)}\rarrow0$, $k=1$, $2$, to exact sequences
$\Psi(T_1^{(k)})\rarrow\Psi(T_0^{(k)})\rarrow\Psi(E^{(k)})\rarrow0$
with the objects $\Psi(T_j^{(k)})$ belonging to $\B_\proj$, so
the groups $\Hom_\B(\Psi(E'),\Psi(E''))$ can be computed similarly in
terms of morphisms between the objects $\Psi(T_j^{(k)})$.
 It remains to recall that the functor $\Psi|_{\Add(T)}$ is fully
faithful (see Section~\ref{co-tilted-categories}).

 Furthermore, since the functor $\Psi|_{\E_\mx(T)}$ is exact and fully
faithful, and takes the projective objects of $\E_\mx(T)$ to projective
objects in $\B$, and since there are enough projectives in $\E_\mx(T)$,
it follows that the functor $\Psi|_{\E_\mx(T)}$ induces isomorphisms
of the Ext groups
$$
 \Ext^i_{\E_\mx(T)}(E',E'')\simeq\Ext^i_\B(\Psi(E'),\Psi(E''))
$$
for all objects $E'$ and $E''\in\E_\mx(T)$ and all $i\ge0$.
 Similarly, as there are enough injectives in $\E_\mx(T)$ and
the injectives of $\E_\mx(T)$ are injective in $\A$, one has
$$
 \Ext^i_{\E_\mx(T)}(E',E'')\simeq\Ext^i_\A(E',E''), \qquad
 E', \ E''\in\E_\mx(T),\ \ i\ge0.
$$
 The functor $\Psi$, being a right adjoint, preserves products; so
the equations $\Prod(J)=\A_\inj$ and $W=\Psi(J)$ imply
$\Prod(W)=\Psi(\A_\inj)$.
 In particular, $W^X=\Psi(J^X)$ for any set~$X$.
 As $\A_\inj\subset\E_\mx(T)$ and $\Ext_\A^i(J^X,J)=0$ for $i>0$,
it follows that $\Ext_\B^i(W^X,W)=0$. 
 So the object $W\in\B$ is weakly cotilting.

 Moreover, for the same reasons one has $\Ext^i_\B(\Psi(E),W)=0$
for all $E\in\E_\mx(T)$ and $i>0$.
 In other words, the objects $\Psi(E)\in\B$ satisfy
the condition~(i$_\mx^{\textstyle*}$).
 Let us show that they also satisfy~(ii$_\mx^{\textstyle*}$),
that is $\Psi(\E_\mx(T))\subset\F_\mx(W)$.
 Let $0\rarrow E\rarrow J^0\rarrow J^1\rarrow J^2\rarrow\dotsb$
be an injective coresolution of $E$ in~$\A$.
 In view of Lemma~\ref{wakamatsu-lemma}(b), this coresolution is
an acyclic complex in the exact category $\E_\mx(T)$.
 The object $J\in\A$ being injective, this coresolution is taken to
an acyclic complex of abelian groups by the contravariant
functor $\Hom_\A({-},J)$.
 Hence, applying the fully faithful exact functor $\Psi|_{\E_\mx(T)}$,
we obtain a coresolution~(ii$_\mx^{\textstyle*}$)
for the object $\Psi(E)$.
 Thus $\B_\proj=\Psi(\Add(T))\subset\Psi(\E_\mx(T))\subset
\F_\mx(W)$, and we have shown that the object $W$ is
$\infty$\+cotilting in~$\B$.

 There are enough injective objects in the category $\A$, and
the left exact functor $\Psi$ establishes an equivalence
$\A_\inj\simeq\Prod(W)$.
 Hence we have $\A=\pi_W(\B)$.
 The assertions dual to what we have already proved now tell that
the functor $\Phi$ is exact and fully faithful in restriction to
$\F_\mx(W)$ and that $\Phi(\F_\mx(W))\subset\E_\mx(T)$.
 Being an adjoint pair of exact and fully faithful functors,
$\Psi|_{\E_\mx(T)}$ and $\Phi|_{\F_\mx(W)}$ are equivalences
of the exact categories $\E_\mx(T)$ and $\F_\mx(W)$.
\end{proof}

\begin{thm} \label{cotilting-tilting}
 Let\/ $\B$ be a complete, cocomplete abelian category with
a projective generator $P$ and an\/ $\infty$\+cotilting object
$W\in\B$.
 Put\/ $\A=\pi_W(\B)$, and let\/ $\Psi\:\A\rarrow\B$ be the left
exact functor identifying the full subcategory of injective
objects\/ $\A_\inj\subset\A$ with the full subcategory\/
$\Prod(W)\subset\B$.
 Let\/ $\Phi\:\B\rarrow\A$ be the right exact functor left adjoint
to\/ $\Psi$; so $J=\Phi(W)$ is a injective cogenerator of\/~$\A$.
 Set $T=\Phi(P)\in\A$.

 Then $T$ is an\/ $\infty$\+tilting object in\/~$\A$,
and the restrictions of the functors\/ $\Phi$ and\/ $\Psi$
induce a pair of inverse equivalences of exact categories between
$\F_\mx(W)$ and\/ $\E_\mx(T)$ (see Figure~\ref{tilting-cotilting-fig}),
which identify the\/ $\infty$\+cotilting object $W\in\B$ with
the injective cogenerator $J\in\A$ and the\/ $\infty$\+tilting
object $T\in\A$ with the projective generator $P\in\B$.
%
 %Then $T$ is an\/ $\infty$\+tilting object in\/~$\A$.
 %The restrictions of the functors\/ $\Phi$ and\/ $\Psi$ to
%the full subcategories\/ $\F_\mx(W)$ and\/ $\E_\mx(T)$ are exact
%and fully faithful, and take them into each other.
 %The two adjoint functors\/ $\Phi|_{\F_\mx(W)}\:\F_\mx(W)\rarrow
%\E_\mx(T)$ and\/ $\Psi|_{\E_\mx(T)}\:\E_\mx(T)\rarrow\F_\mx(W)$
%are mutually inverse equivalences of exact categories taking
%the\/ $\infty$\+cotilting object $W\in\B$ to the injective cogenerator
%$J\in\A$ and the\/ $\infty$\+tilting object $T\in\A$ to
%the projective generator $P\in\B$.
\end{thm}

\begin{proof}
 Dual to Theorem~\ref{tilting-cotilting}.
\end{proof}

\begin{cor} \label{tilting-cotilting-cor}
 The constructions of Theorems~\textup{\ref{tilting-cotilting}}
and~\textup{\ref{cotilting-tilting}} establish a one-to-one
correspondence between equivalence classes of
\begin{enumerate}
\item complete, cocomplete abelian categories\/ $\A$
with an injective cogenerator $J$ and an\/ $\infty$\+tilting object $T$,
and
\item complete, cocomplete abelian categories\/ $\B$ with a projective
generator $P$ and an\/ $\infty$\+cotilting object~$W$.  \qed
\end{enumerate}
\end{cor}

\Section{\texorpdfstring{$\infty$}{Infinity}-Tilting and \texorpdfstring{$\infty$}{Infinity}-Cotilting Pairs}
%\Section{$\infty$-Tilting and $\infty$-Cotilting Pairs}
\label{infty-til-cotil-pairs-secn}

 As above, let $\A$ be an abelian category with set-indexed products and
an injective cogenerator $J\in\A$.
 Let $T\in\A$ be an object and $\E\subset\A$ be a full subcategory.
 We will say that $(T,\.\E)$ is an \emph{$\infty$\+tilting pair} in $\A$
if the following conditions hold:
\begin{enumerate}
\renewcommand{\theenumi}{\roman{enumi}}
\item $\A_\inj\subset\E$;
\item $\Add(T)\subset\E$;
\item $\Ext_\A^1(T,E)=0$ for all $E\in\E$;
\item $\E$ is closed under the cokernels of monomorphisms and
extensions in~$\A$;
\item every $\Add(T)$\+precover $T'\rarrow E$ of an object $E\in\E$
is an epimorphism in $\A$ with the kernel belonging to~$\E$.
\end{enumerate}

 Due to the condition~(iv), the full subcategory $\E\subset\A$ inherits
an exact category structure from the abelian category~$\A$.
 According to the condition~(i), there are enough injective objects
in the exact category $\E$, and these are precisely the injective
objects of the ambient abelian category $\A$, that is $\E_\inj=\A_\inj$.

 It follows from the condition~(iii) together with the condition~(i)
and the first part of the condition~(iv) that
$$
 \Ext_\A^i(T,E)=0 \quad\text{for all $E\in\E$ and
 all integers $i>0$}.
$$
 Hence, in view of the condition~(ii), the object $T\in\A$ has to be 
weakly tilting.

 From the conditions~(ii) and~(iii) we see that the objects of
$\Add(T)$ are projective in the exact category~$\E$.
 It follows from the condition~(v) that there are enough projective
objects belonging to $\Add(T)$ in $\E$.
 Hence there are enough projective objects in $\E$ and
the class of all projective objects in $\E$ coincides with $\Add(T)$,
that is $\E_\proj=\Add(T)$.

 Now it is clear that all the objects $E\in\E$ satisfy
the conditions~(i$_\mx$) and~(ii$_\mx$); so we have
$\E\subset\E_\mx(T)\subset\A$.
 From the condition~(i) we conclude that $\A_\inj\subset\E_\mx(T)$.
 Thus the object $T\in\A$ has to be $\infty$\+tilting.
 Conversely, according to Lemma~\ref{wakamatsu-lemma}, for
any $\infty$\+tilting object $T\in\A$ the pair $(T,\.\E_\mx(T))$
is an $\infty$\+tilting pair in~$\A$.
To summarize, we have shown the following.

\begin{lem} \label{tilting-objects-vs-tilting-pairs}
Let\/ $\A$ be a complete, cocomplete abelian category with
an injective cogenerator. Then an object $T\in\A$ is a part of
an $\infty$\+tilting pair $(T,\.\E)$ in\/ $\A$ if and only if it is
an $\infty$\+tilting object.
 The full subcategory\/ $\E=\E_\mx(T)$ is the maximal of all
full subcategories\/ $\E\subset\A$ forming an $\infty$\+tilting
pair with $T\in\nobreak\A$.
\qed
\end{lem}

In general, we do not assume that $\E$ is closed under direct summands.
However, we can add that assumption whenever convenient (e.g.\ in
Sections~\ref{derived-equivalence-secn} or~\ref{t-structures-secn}):

\begin{lem} \label{Karoubian-closure-of-tilting-pair}
If $(T,\.\E)$ is an $\infty$\+tilting pair in\/ $\A$ and\/ $\E'$ is
the closure of\/ $\E$ under direct summands, then $(T,\.\E')$ is also
an $\infty$\+tilting pair and\/ $\E'$ is a coresolving subcategory
in\/~$\A$.
\end{lem}

\begin{proof}
The conditions~(i--iii) are obviously true for $\E'$.
To prove~(iv), suppose that we have an exact sequence
$0 \rarrow E'_1\rarrow E_1\rarrow E''_1\rarrow 0$ with
$E'_1$, $E''_1\in \E'$, i.e.\ there exist $E'_2$, $E''_2\in\A$
such that $E'=E'_1\oplus E'_2$ and $E''=E''_1\oplus E''_2$
belong to~$\E$.
Then $E_1\oplus E'_2\oplus E''_2$ is an extension of $E'$ by $E''$
in $\A$, and hence $E_1\in\E'$.
Similarly, if $f_1\: E'_1 \rarrow E_1$ is a monomorphism in $\A$
with $E'_1$, $E_1\in\E'$, then there is a split monomorphism
$f_2\: E'_2\rarrow E_2$ such that $f_1\oplus f_2$ is a monomorphism
in $\A$ between objects of~$\E$.
Finally, to prove~(v), it suffices to note that if $E = E_1 \oplus E_2
\in\E$ and if $t_1\: T_1 \rarrow E_1$ and $t_2\: T_2 \rarrow E_2$
are $\Add(T)$\+precovers, then also $t_1\oplus t_2\: T_1\oplus T_2
\rarrow E$ is an $\Add(T)$\+precover.
\end{proof}

\medskip

 Now we present the dual definitions.
 Let $\B$ be an abelian category with set-indexed coproducts and
a projective generator $P\in\B$.
 Let $W\in\B$ be an object and $\F\subset\B$ be a full subcategory.
 We will say that $(W,\.\F)$ is an \emph{$\infty$\+cotilting pair}
in $\B$ if the following conditions hold:
\begin{enumerate}
\renewcommand{\theenumi}{\roman{enumi}*}
\item $\B_\proj\subset\F$;
\item $\Prod(W)\subset\F$;
\item $\Ext_\B^1(F,W)=0$ for all $F\in\F$;
\item $\F$ is closed under the kernels of epimorphisms and
extensions in~$\B$;
\item every $\Prod(W)$\+preenvelope $F\rarrow W'$ of an object $F\in\F$
is a monomorphism in $\B$ with the cokernel belonging to~$\F$.
\end{enumerate}

 As above, it follows from the conditions~(i*--v*) that 
$$
 \Ext_\B^i(F,W)=0 \quad\text{for all $F\in\F$ and all integers $i>0$},
$$
the object $W\in\B$ is weakly cotilting, and the full subcategory
$\F\subset\B$ inherits an exact category structure from
the abelian category~$\B$.
 The exact category $\F$ has both enough projective and enough
injective objects; the full subcategories of projective and
injective objects in $\F$ are described as $\F_\proj=\B_\proj$ and
$\F_\inj=\Prod(W)$. Moreover, as before one also has:

\begin{lem} \label{cotilting-objects-cotilting-pairs-Karoubian-closure}
Let\/ $\B$ be a complete, cocomplete abelian category with
a projective generator. Then an object $W\in\B$ is a part of
an $\infty$\+cotilting pair $(W,\.\F)$ in\/ $\B$ if and only if it is
an $\infty$\+cotilting object.
The full subcategory\/ $\F=\F_\mx(W)$ is the maximal of all full
subcategories\/ $\F\subset\B$ forming an $\infty$\+cotilting pair
with $W\in\B$.

Moreover, if $(W,\.\F)$ is an $\infty$\+cotilting pair and\/ $\F'$ is
the closure of\/ $\F$ under direct summands, then $(W,\.\F')$ is also
an $\infty$\+cotilting pair and\/ $\F'$ is a resolving subcategory
of\/~$\B$.
\end{lem}

\begin{proof}
This is dual to Lemmas~\ref{tilting-objects-vs-tilting-pairs}
and~\ref{Karoubian-closure-of-tilting-pair}.
\end{proof}

The $\infty$-tilting-cotilting correspondence from the last section now
extends to one between $\infty$-tilting and $\infty$-cotilting pairs.

\begin{prop} \label{tilting-cotilting-pairs-prop}
 In the context of Corollary~\textup{\ref{tilting-cotilting-cor}}
(see also Figure~\ref{tilting-cotilting-fig}),
the assignments\/ $\F=\Psi(\E)$ and\/ $\E=\Phi(\F)$ establish
a bijective correspondence between
\begin{enumerate}
\item the full subcategories\/ $\E\subset\E_\mx(T)$ forming
an\/ $\infty$\+tilting pair with $T\in\A$ and
\item the full subcategories\/ $\F\subset\F_\mx(W)$
forming an\/ $\infty$\+cotilting pair with $W\in\B$.
\end{enumerate}
\end{prop}

\begin{proof}
 Let $(T,\.\E)$ be an $\infty$\+tilting pair in the category~$\A$.
 Put $\F=\Psi(\E)$.
 We have to show that $(W,\.\F)$ is an $\infty$\+cotilting pair
in the category~$\B$.

 Indeed, the condition~(i*) follows from~(ii) and the condition~(ii*)
follows from~(i), as $\B_\proj=\Psi(\Add(T))$ and
$\Prod(W)=\Psi(\A_\inj)$.
 The condition~(iii*) holds, since $\F=\Psi(\E)\subset\Psi(\E_\mx(T))
=\F_\mx(W)$ and all $F\in\F_\mx(W)$ satisfy~(iii*).

 The full subcategory $\F$ is closed under extensions in $\F_\mx(W)$,
since $\Psi\:\E_\mx(T)\rarrow\F_\mx(W)$ is an equivalence of exact
categories and the full subcategory $\E$ is closed under extensions
in $\E_\mx(W)$.
 Since the full subcategory $\F_\mx(W)$ is closed under extensions
in $\B$ by Lemma~\ref{wakamatsu-colemma}(a), it follows that
$\F$ is closed under extensions in~$\B$.

 Let $f\:F'\rarrow F''$ be an epimorphism in $\B$ between two objects
$F'$, $F''\in\F$.
 Then there exists a morphism $e\:E'\rarrow E''$ in $\E$ such that
$F^{(s)}\simeq\Psi(E^{(s)})$, \,$s=1$,~$2$, and $f=\Psi(e)$.
 The map of abelian groups $\Hom_\E(T,e)$ is surjective, since the map
$\Hom_\F(P,f)$ is and $P=\Psi(T)$.
 Let $T_0\rarrow E'$ be an $\Add(T)$\+precover; then the composition
$T_0\rarrow E'\rarrow E''$ is also an $\Add(T)$\+precover.
 Denote the kernels of the morphisms $T_0\rarrow E'$ and
$T_0\rarrow E''$ by $Z'_0$ and~$Z''_0$, respectively.
 Then $Z'_0$, $Z''_0\in\E$ by the condition~(v) and the natural morphism
$Z'_0\rarrow Z''_0$ is a monomorphism.
 Hence $\ker(e)=\coker(Z'_0\to Z''_0)\in\E$ by the condition~(iv) and
$\ker(f)=\Psi(\ker e)\in\F$.
 This proves that the full subcategory $\F\subset\B$ is closed under
the kernels of epimorphisms in~$\B$ and finishes the proof of 
the condition~(iv*).

 To prove the condition~(v*), let $f\:F\rarrow W^0$ be
a $\Prod(W)$\+preenvelope of an object $F\in\F$.
 Then $f=\Psi(e)$, where $e\:E\rarrow J^0$ is a morphism in $\E$
and $J^0\in\A_\inj$.
 The map $\Hom_\E(e,J)$ is surjective, since the map $\Hom_\F(f,W)$ is
and $W=\Psi(J)$.
 Since $J$ is an injective cogenerator of $\A$, it follows that $e$~is
a monomorphism in~$\A$.
 By the condition~(iv), the cokernel of~$e$ belongs to~$\E$, so
$e$~is a monomorphism in~$\E$.
 Since the functor $\Psi|_\E$ is exact, it follows that $f$~is
a monomorphism in $\B$ with the cokernel belonging to~$\F$.

 To summarize these arguments, the conditions~(iv\+v) essentially
say that the full subcategory $\E$ is closed under the cokernels of
monomorphisms, extensions, and kernels of epimorphisms in
$\E_\mx(T)$, while the conditions~(iv*\+v*) mean that the full
subcategory $\F$ is closed under the kernels of epimorphisms,
extensions, and cokernels of monomorphisms in $\F_\mx(W)$. 
\end{proof}

\begin{cor} \label{tilting-cotilting-correspondence-cor}
 The constructions of
Theorems~\textup{\ref{tilting-cotilting}}\+-%
\textup{\ref{cotilting-tilting}} and
Proposition~\textup{\ref{tilting-cotilting-pairs-prop}} establish
a one-to-one correspondence between equivalence classes of
\begin{enumerate}
\item quadruples $(\A,\E,T,J)$, where\/ $\A$ is a complete, cocomplete
abelian category with an injective cogenerator $J$ and
$(T,\.\E)$ is an\/ $\infty$\+tilting pair in\/~$\A$, and
\item quadruples $(\B,\F,P,W)$, where\/ $\B$ is a complete, cocomplete
abelian category with a projective generator $P$ and
$(W,\.\F)$ is an\/ $\infty$\+cotilting pair in\/~$\B$.
\end{enumerate}
In this correspondence, the exact categories\/ $\E$ and\/ $\F$ are
naturally equivalent, $\E\simeq\F$, and the equivalence identifies
$T$ with $P$ and $W$ with $J$.  \qed
\end{cor}

In general, there can be many classes $\E$ which form an 
$\infty$-tilting pair with a given $\infty$-tilting object
$T \in \A$. Thanks to the following lemma, we know that
they form a complete lattice.

\begin{lem} \label{lattice-of-tilting-classes}
Let\/ $\A$ be a complete, cocomplete abelian category with
an injective cogenerator and let $T\in\A$ be an $\infty$\+tilting object.
If\/ $\E_i\subset\A$, $i\in I$, is a collection of full subcategories
such that $(T,\.\E_i)$ is an $\infty$\+tilting pair for each $i\in I$,
then $(T,\.\E)$ is an $\infty$\+tilting pair
with\/ $\E=\bigcap_{i\in I}\E_i$.

Dually, if\/ $\B$ is a complete, cocomplete abelian category with
a projective generator, $W\in\B$ is an $\infty$\+cotilting object and
$(W,\.\F_j)$ are $\infty$\+cotilting pairs, $j\in J$, then $(W,\.\F)$
is an $\infty$\+cotilting pair with\/ $\F=\bigcap_{j\in J}\F_j$.
\end{lem}

\begin{proof}
It is straightforward to check that each of the conditions~(i--v)
and~(i*--v*) is preserved by intersections of classes.
\end{proof}

\begin{ex} \label{minimal-pairs-ex}
 In particular, whenever $T$ is an $\infty$\+tilting object
in~$\A$, there exists a unique minimal full subcategory
$\E_\mn(T)\subset\A$ for which $(T,\.\E_\mn(T))$ is an
$\infty$\+tilting pair in~$\A$.
 In fact, the full subcategory $\E_\mn(T)$ consists of all the
objects in $\A$ that can be obtained from the objects of
$\A_\inj\subset \E_\mn(\A)$ and $\Add(T)\subset\E_\mn(\A)$ by
applying iteratively the operations of the passage to
the cokernel of a monomorphism, an extension, or the kernel
of an $\Add(T)$\+precover.
 For every $\infty$\+tilting pair $(T,\.\E)$ in $\A$, one then has
$\E_\mn(T)\subset\E$.

Similarly, whenever $W$ is an $\infty$\+cotilting
object in~$\B$, there exists a unique minimal full subcategory
$\F_\mn(W)\subset\B$ such that $(W,\.\F_\mn(W))$ is
an $\infty$\+cotilting pair in~$\B$.
 For every $\infty$\+cotilting pair $(W,\.\F)$ in $\B$, one has
$\F_\mn(W)\subset\F$.

 In the situation of Corollary~\ref{tilting-cotilting-cor}
(and Figure~\ref{tilting-cotilting-fig}), the full
subcategories $\E_\mn(T)\subset\A$ and $\F_\mn(W)\subset\B$
are transformed into each other by the functors $\Psi$ and $\Phi$,
that is $\F_\mn(W)=\Psi(\E_\mn(T))$ and $\E_\mn(T)=\Phi(\F_\mn(W))$.
\end{ex}

\begin{rem}
 There is a certain similarity between our results in
Sections~\ref{co-tilted-categories}\+-\ref{infty-til-cotil-pairs-secn}
of this paper and those in the recent paper~\cite[Sections~2\+-3]{En}.
 Let us explain the connection and the differences between our
approaches.
 The paper~\cite{En} is a far-reaching development of the traditional
point of view in Wakamatsu tilting theory, in which finitely generated
modules over Artinian algebras are the main objects of study.
 The author of~\cite{En} works with skeletally small exact categories,
and essentially never considers infinite products or coproducts.
 The definition of a projective generator in~\cite[paragraph before
Corollary~2.14]{En} presumes an exact category with enough
projective objects in which every projective object is a direct summand
of a \emph{finite} direct sum of copies of the (single) generator.

 Nevertheless, the generality level in~\cite[Sections~2\+-3]{En}
exceeds that of our exposition.
 In particular, our Lemma~\ref{wakamatsu-colemma} is but
a particular case of~\cite[Proposition~3.2]{En} (while our
Lemma~\ref{wakamatsu-lemma} is dual).
 The author of~\cite{En} achieves this generality by working with
arbitrary (skeletally small) additive categories in place of our classes
$\Add(T)$ and $\Prod(W)$.
 An exact category playing the role of our $\F$ is generally denoted
by~$\mathcal E$ in~\cite{En}, an additive category playing the role
of our $\Add(T)=\F_\proj=\B_\proj$ is denoted by $\mathcal C$,
an additive category in the role of our $\Prod(W)=\F_\inj$
is denoted by $\mathcal W$, and the exact category in the role
of our $\F_\mx(W)$ is denoted by~$\mathsf X_{\mathcal W}$.
 (The reader should be warned that the author of~\cite{En} calls
``Wakamatsu tilting'' what we would call ``Wakamatsu cotilting''
or ``$\infty$\+cotilting''.)

 Finally, in the role of our abelian category $\B$, the author
of~\cite{En} has an exact category which he denotes by
$\mathsf{mod}\,\mathcal C$.
 This difference occurs because our observation that the category
$\Add(T)$ always has weak kernels (as pointed out in
the proof of Theorem~\ref{add-prod-theorem}) has
no counterpart in~\cite{En}.
\end{rem}

\Section{\texorpdfstring{$\infty$}{Infinity}-Tilting-Cotilting Derived Equivalences}
%\Section{$\infty$-Tilting-Cotilting Derived Equivalences}
\label{derived-equivalence-secn}

\begin{figure}
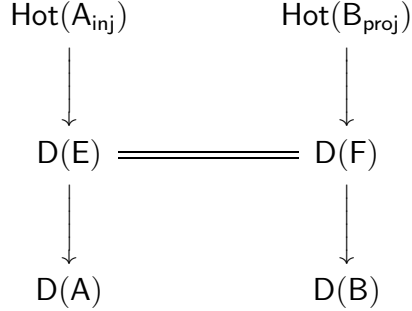

\[
\begin{diagram}
\node{\Hot(\A_\inj)}\arrow{s}\node[2]{\Hot(\B_\proj)}\arrow{s} \\
\node{\D(\E)}\arrow{s}\arrow[2]{e,=}\node[2]{\D(\F)}\arrow{s} \\
\node{\D(\A)}\node[2]{\D(\B)}
\end{diagram}
\]
\caption{Equivalences induced by $\infty$-tilting/$\infty$-cotilting pairs in general.}%
\label{equivalences-hotinj-fig}%
\end{figure}

 Unlike in~\cite[Sections~3\+-5]{PS}, in our present situation
the coresolution dimensions of objects of the category $\A$ with
respect to its coresolving subcategory $\E$ or $\E_\mx(T)$ can well be
infinite, and so can the resolution dimensions of objects of the category
$\B$ with respect to its resolving subcategory $\F$ or $\F_\mx(W)$.
 Hence the equivalence of exact categories $\E\simeq\F$ does not
generally lead to any equivalence between the derived categories
$\D(\A)$ and~$\D(\B)$.
 All one can say is that there is the commutative diagram
formed by triangulated functors and a triangulated equivalence in
Figure~\ref{equivalences-hotinj-fig}.
 Here $\Hot(\A_\inj)$ and $\Hot(\B_\proj)$ are the homotopy categories
of (unbounded complexes in) the additive categories $\A_\inj$
and~$\B_\proj$, while $\D(\E)$ and $\D(\F)$ are the (unbounded) derived
categories of the exact categories $\E$ and $\F$, and $\D(\A)$ and
$\D(\B)$ are the similar derived categories of the abelian
categories $\A$ and~$\B$.

 If $\A$ is a Grothendieck category, the canonical functor
$\Hot(\A_\inj) \rarrow \D(\A)$ in the left-hand side column
of the above diagram is a Verdier quotient functor.
This follows e.g.\ from~\cite[Theorem~5.4]{AJS}.
 If the full subcategories $\E\subset\A$ and $\F\subset\B$
have additional closure properties, we will obtain a similar diagram
below where all the functors are Verdier quotients. 

One issue here is that, unlike for tilting modules of finite projective
dimension, the class $\E$ in the definition of a tilting pair need not
be closed under coproducts in $\A$ (cf.~\cite[Lemma~5.3]{PS}).
Dually, the class $\F$ need not be closed under products.
There are some elementary relations between the closure properties
of $\E$ and $\F$, however.

\begin{lem} \label{co-products-preservation}
 In the context of
Corollary~\textup{\ref{tilting-cotilting-correspondence-cor}},
if the full subcategory\/ $\E\subset\A$ is closed under products,
then the full subcategory\/ $\F\subset\B$ is closed under products.
 If the full subcategory\/ $\F\subset\B$ is closed under coproducts,
then the full subcategory\/ $\E\subset\A$ is closed under coproducts.
\end{lem}

\begin{proof}
 The first assertion holds, since $\F=\Psi(\E)$ and the functor
$\Psi\:\A\rarrow\B$ preserves products (see Figure~\ref{tilting-cotilting-fig}).
 The second assertion holds, since $\E=\Phi(\F)$ and the functor
$\Phi\:\B\rarrow\A$ preserves coproducts.
\end{proof}

 It would be interesting to know whether the converse assertions
to those of Lemma~\ref{co-products-preservation} are true.

Suppose now that $\E$ is a part of an $\infty$\+tilting pair in
a complete, cocomplete abelian category $\A$ with an injective cogenerator.
Then $\A$ has exact coproducts (\cite[Exercise~III.2]{Mit}) and,
if $\E$ is closed under coproducts,
$\E$ has exact coproducts too. In such a situation the following definition
from~\cite[Sections~2.1 and~4.1]{Psemi} or~\cite[Section~A.1]{Pcosh}
applies and gives a more adequate replacement of
$\Hot(\A_\inj) = \Hot(\E_\inj)$ in Figure~\ref{equivalences-hotinj-fig}.

If $\E$ is an exact category with arbitrary coproducts which are exact,
we call a complex \emph{coacyclic} if it belongs to the smallest
localizing subcategory of $\Hot(\E)$ which contains the total complexes
of short exact sequences of complexes over~$\E$.
The \emph{coderived category} of $\E$, which we denote by
$\D^\co(\E)$, is defined as the Verdier quotient category of $\Hot(\E)$ by
the subcategory of coacyclic complexes.

Note that it follows from the above definition that each coacyclic
complex is exact and, thus, we have a Verdier quotient functor
$\D^\co(\E) \rarrow \D(\E)$. On the other hand, if $\E$ in addition
has enough injectives, the natural functor
$\Hot(\E_\inj) \rarrow \D^\co(\E)$ is fully faithful
by~\cite[Lemma A.1.3]{Pcosh}.
To summarize, we have triangulated functors
\[
\xymatrix@1{ \Hot(\E_\inj) \; \ar@{>->}[r] & \; \D^\co(\E) \; \ar@{->>}[r] & \; \D(\E), }
\]
where the first functor is fully faithful and the second one is
a Verdier quotient. In fact, the fully faithful functor was proved to be
an equivalence in some cases~\cite[Theorem~2.4]{PcohSGdual}.

If $\F$ is an exact category with arbitrary products which are exact,
the class of \emph{contraacyclic} complexes in $\Hot(\F)$ and
the \emph{contraderived} category $\D^\ctr(\F)$ of $\F$ are defined
dually, and we have triangulated functors
\[
\xymatrix@1{ \Hot(\F_\proj) \; \ar@{>->}[r] & \; \D^\ctr(\F) \; \ar@{->>}[r] & \; \D(\F). }
\]
As above, the fully faithful functor in the leftmost arrow is known to be
an equivalence in some cases~\cite[Theorem~4.4(b)]{PcohSGdual}.

Now we can state the main result of the section (see also
Figure~\ref{equivalences-coderived-fig} below).

\begin{prop} \label{co-products-closed-exact-subcategory}
\textup{(a)} Let\/ $\A$ be an exact category where set-indexed
coproducts exist and are exact, and let\/ $\E\subset\A$ be
a coresolving subcategory closed under coproducts.
 Then the functor between the coderived categories\/
$\D^\co(\E)\rarrow\D^\co(\A)$ induced by the embedding of exact
categories\/ $\E\rarrow\A$ is a triangulated equivalence.
 The triangulated functor between the conventional derived
categories\/ $\D(\E)\rarrow\D(\A)$ induced by the same exact
embedding is a Verdier quotient functor. \par
\textup{(b)} Let\/ $\B$ be an exact category where set-indexed
products exist and are exact, and let\/ $\F\subset\B$ be
a resolving subcategory closed under products.
 Then the functor between the contraderived categories\/
$\D^\ctr(\F)\rarrow\D^\ctr(\B)$ induced by the embedding of exact
categories\/ $\F\rarrow\B$ is a triangulated equivalence.
 The triangulated functor between the conventional derived
categories\/ $\D(\F)\rarrow\D(\B)$ induced by the same exact
embedding is a Verdier quotient functor.
\end{prop}

\begin{proof}
 The first assertion of part~(b) is~\cite[Proposition~A.3.1(b)]{Pcosh},
and the first assertion of part~(a) is the dual result.

 To prove the second assertion of part~(a), notice that we have
a commutative diagram of triangulated functors
$\D^\co(\E)=\D^\co(\A)\rarrow\D(\E)\rarrow\D(\A)$, where both
the functors $\D^\co(\A)\rarrow\D(\E)$ and $\D^\co(\A)\rarrow
\D(\A)$ are Verdier quotient functors.
 It follows that the functor $\D(\E)\rarrow\D(\A)$ is also
a Verdier quotient.
\end{proof}

\begin{figure}
\[
\begin{diagram}
\node{\D^\co(\A)}\arrow{s,A}\node[2]{\D^\ctr(\B)}\arrow{s,A} \\
\node{\D(\E)}\arrow{s,A}\arrow[2]{e,=}\node[2]{\D(\F)}\arrow{s,A} \\
\node{\D(\A)}\node[2]{\D(\B)}
\end{diagram}
\]
\caption{Equivalences induced by $\infty$-tilting/$\infty$-cotilting pairs when $\E$ is closed under coproducts and $\F$ under products.}%
\label{equivalences-coderived-fig}%
\end{figure}

 In particular, Proposition~\ref{co-products-closed-exact-subcategory}
tells that, when in the situation of 
Corollary~\ref{tilting-cotilting-correspondence-cor} the full
subcategory $\E\subset\A$ is closed under coproducts and the full
subcategory $\F\subset\B$ is closed under products, we have
a commutative diagram formed by Verdier quotient functors and
a triangulated equivalence as in Figure~\ref{equivalences-coderived-fig}.

\begin{rem}
 Let $\A$ be an exact category with exact coproducts, and let
$\E'\subset\E''\subset\A$ be two coresolving subcategories closed
under coproducts.
 Then one has $\D^\co(\E')\simeq\D^\co(\E'')\simeq\D^\co(\A)$,
while the natural functors between the conventional derived categories
$\D(\E')\rarrow\D(\E'')\rarrow\D(\A)$ are Verdier quotient functors.
 Thus, when a coproduct-closed coresolving subcategory is being
enlarged, its derived category gets deflated.
 In other words, the larger the subcategory $\E\subset\A$, the smaller
its derived category $\D(\E)$.

 Similarly, let $\B$ be an exact category with exact products, and
let $\F'\subset\F''\subset\B$ be two resolving subcategories closed
under products.
 Then one has $\D^\ctr(\F')\simeq\D^\ctr(\F'')\simeq\D^\ctr(\B)$,
while the natural functors between the conventional derived categories
$\D(\F')\rarrow\D(\F'')\rarrow\D(\B)$ are Verdier quotient functors.

 In particular, when in the situation of
Proposition~\ref{tilting-cotilting-pairs-prop} there are
two $\infty$\+tilting pairs $(T,\.\E')$ and $(T,\.\E'')$ with
$\E'\subset\E''\subset\A$, and the corresponding
two $\infty$\+cotilting pairs are $(W,\.\F')$ and $(W,\.\F'')$,
so $\F'\subset\F''\subset\B$, we obtain the commutative diagram of
Verdier quotient functors and triangulated equivalences as in
Figure~\ref{equivalences-compatible-fig}.
  We refer to~\cite[Section~1]{Pps} for a further discussion.
\end{rem}

\begin{figure}
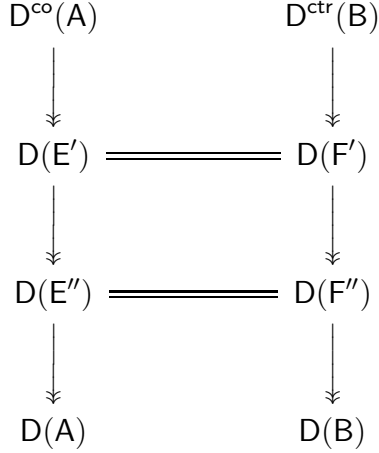

\[
\begin{diagram}
\node{\D^\co(\A)}\arrow{s,A}\node[2]{\D^\ctr(\B)}\arrow{s,A} \\
\node{\D(\E')}\arrow{s,A}\arrow[2]{e,=}\node[2]{\D(\F')}\arrow{s,A} \\
\node{\D(\E'')}\arrow{s,A}\arrow[2]{e,=}\node[2]{\D(\F'')}\arrow{s,A}
\\ \node{\D(\A)}\node[2]{\D(\B)}
\end{diagram}
\]
\caption{Compatible equivalences for different choices of $\infty$-tilting/$\infty$-cotilting pairs.}%
\label{equivalences-compatible-fig}%
\end{figure}

\Section{\texorpdfstring{$\infty$}{Infinity}-Tilting and \texorpdfstring{$\infty$}{Infinity}-Cotilting \texorpdfstring{$\mathrm t$}{t}-Structures}
%\Section{$\infty$-Tilting and $\infty$-Cotilting $\mathrm t$-Structures}
\label{t-structures-secn}

The aim of the section is to lift the canonical t\+structures from $\D(\A)$ and $\D(\B)$ to $\D(\E)$ and $\D(\F)$, respectively, in Figures~\ref{equivalences-hotinj-fig} or~\ref{equivalences-coderived-fig} in the previous section. By doing this, we obtain a picture very similar to the classical tilting theory, where both $\A$ and $\B$ can be viewed as full subcategories of $\D(\E)$ such that $\E=\A\cap\B$ (since $\E\simeq\F$, we of course obtain the same picture in $\D(\F)$).

 We start with a lemma showing that t\+structures can be lifted with
respect to certain triangulated functors with partial adjoints.

\begin{lem} \label{lifting-t-structure-lemma}
 Let\/ $\D$ and ${}'\D$ be triangulated categories and
$(\D^{\le0},\.\D^{\ge0})$ be a t\+structure on\/~$\D$.
 Let $F\:{}'\D\rarrow\D$ be a triangulated functor such that
a right adjoint functor to $F$ is defined on\/ $\D^{\ge0}\subset\D$,
that is, for every object $X\in\D^{\ge0}$ there exists an object
$G(X)\in{}'\D$ such that the functors\/ $\Hom_{\D}(F({-}),X)$
and\/ $\Hom_{{}'\D}({-},G(X))$ are isomorphic on\/~${}'\D$.
 Assume that the adjunction morphism $\varepsilon_X\colon FG(X)\rarrow X$ is
an isomorphism in\/ $\D$ for all objects $X\in\D^{\ge0}$.

 Set\/ ${}'\D^{\le0}=F^{-1}(\D^{\le0})\subset{}'\D$ to be
the full preimage of\/ $\D^{\le0}$ under $F$ and\/
${}'\D^{\ge0}=G(\D^{\ge0})\subset{}'\D$ to be the essential
image of\/ $\D^{\ge0}$ under~$G$.
 Then the pair of full subcategories $({}'\D^{\le0},\.{}'\D^{\ge0})$
is a t\+structure on\/~${}'\D$.
 The functors $F$ and $G$ restrict to mutually inverse equivalences
between the abelian hearts\/ $\A=\D^{\le0}\cap\D^{\ge0}\subset\D$
and\/ ${}'\!\.\A={}'\D^{\le0}\cap{}'\D^{\ge0}\subset{}'\D$
of the two t\+structures.
\end{lem}

\begin{proof}
 One can easily check that the functor $G$ commutes with the shift
functors $[-1]$ on ${}'\D$ and~$\D$ (since the functor $F$~does).
 Let us show that $\Hom_{{}'\D}({}'\!X,{}'Y)=0$ for all
${}'\!X\in{}'\D^{\le0}$ and ${}'Y\in{}'\D^{\ge1}$.
 Indeed, we have $F({}'\!X)\in\D^{\le0}$ and ${}'Y=G(Y)$ for some
$Y\in\D^{\ge1}$.
 Hence $\Hom_{{}'\D}({}'\!X,{}'Y)=\Hom_{{}'\D}({}'\!X,G(Y))=
\Hom_\D(F({}'\!X),Y)=0$.

 Now let ${}'\!X\in{}'\D$ be an arbitrary object.
 Set $X=F({}'\!X)\in\D$, and consider a distinguished triangle
\begin{equation} \label{downstairs-triangle}
 \tau_{\le0}X\lrarrow X\lrarrow \tau_{\ge1}X\lrarrow(\tau_{\le0}X)[1]
\end{equation}
in $\D$ with $\tau_{\le0}X\in\D^{\le0}$ and $\tau_{\ge1}X\in\D^{\ge1}$.
 Put $\tau_{\ge1}{}'\!X=G(\tau_{\ge1}X)\in{}'\D^{\ge1}$.
 Then the morphism $F({}'\!X)=X\rarrow\tau_{\ge1}X$ in $\D$
corresponds to a certain morphism ${}'\!X\rarrow G(\tau_{\ge1}X)
=\tau_{\ge1}{}'\!X$ in~${}'\D$.
 Denote by $\tau_{\le0}{}'\!X$ a cocone of the latter morphism, so that
we have a distinguished triangle
\begin{equation} \label{upstairs-triangle}
 \tau_{\le0}{}'\!X\lrarrow{}'\!X\lrarrow\tau_{\ge1}{}'\!X\lrarrow
 (\tau_{\le0}{}'\!X)[1]
\end{equation}
in~${}'\D$.
 Applying the functor $F$ to the morphism ${}'\!X\rarrow
\tau_{\ge1}{}'\!X$ produces the morphism $X=F({}'\!X)\rarrow
F(\tau_{\ge1}({}'\!X))=FG(\tau_{\ge1}(X))=\tau_{\ge1}(X)$.
 Thus the triangulated functor $F$ takes the distinguished
triangle~\eqref{upstairs-triangle} to the distinguished
triangle~\eqref{downstairs-triangle}, and it follows that
the object $F(\tau_{\le0}{}'\!X)$ is isomorphic to~$\tau_{\le0}X$.
%Therefore, $\tau_{\le0}{}'\!X\in{}'\D^{\le0}$.
In other words, we have $\tau_{\le0}{}'\!X\in{}'\D^{\le0}$ and
$\tau_{\ge1}{}'\!X\in{}'\D^{\ge1}:={}'\D^{\ge0}[-1]$
in~\eqref{upstairs-triangle}.
It follows that $({}'\D^{\le0},\.{}'\D^{\ge0})$ is a t\+structure.

 Furthermore, the functors $F$ and $G$ restrict to an equivalence
between the coaisles $\D^{\ge0}\subset\D$ and ${}'\D^{\ge0}\subset{}'\D$.
 Indeed, if $'\!X \in{}'\D^{\ge0}$, then $'\!X = G(X)$ for some
$X\in\D^{\ge0}$ and $F('\!X) = FG(X) = X \in \D^{\ge0}$.
 Thus the functor $F$ restricts to $F\:{}'\D^{\ge0}\rarrow\D^{\ge0}$,
the functor $G\:\D^{\ge0}\rarrow{}'\D^{\ge0}$ is its (honest)
right adjoint, and the composition
\[
 \D^{\ge0}\overset{G}\lrarrow{}'\D^{\ge0}\overset{F}\lrarrow\D^{\ge0}
\]
is the identity functor by assumption.
 Hence the functor $G$ is fully faithful; and its essential image coincides
with ${}'\D^{\ge0}$ by the definition.
 Finally, for any ${}'\!X\in{}'\D^{\ge0}$ we have ${}'\!X\in{}'\!\.\A$
if and only if $F({}'\!X)\in\A$, because we have ${}'\!X\in{}'\D^{\le0}$
if and only if $F({}'\!X)\in\D^{\le0}$.
\end{proof}

\begin{rem}
In the special case where the functor $F\:{}'\D\rarrow\D$ from the former lemma is a part of a recollement

\[
\xymatrix{ \D \ar@{>->}@/^2ex/@<1ex>[rr] \ar@{>->}@/_2ex/@<-1ex>[rr]_-G \ar@{<-}[rr]|-F && {}'\D \ar@/^2ex/@<1ex>[rr] \ar@/_2ex/@<-1ex>[rr] \ar@{<-<}[rr] && {}''\D, }
\]

\noindent
the t\+structure $({}'\D^{\le0},\.{}'\D^{\ge0})$ coincides with the result of gluing $(\D^{\le0},\.\D^{\ge0})$ with the trivial t\+structure $({}''\D,0)$ on ${}''\D$ in the sense of \cite[Th\'eor\`eme 1.4.10]{BBD}.
\end{rem}

 We recall that for any t\+structure $(\D^{\le0},\.\D^{\ge0})$ on
a triangulated category $\D$ with the abelian heart
$\A=\D^{\le0}\cap\D^{\ge0}\subset\D$ there are natural maps
\begin{equation} \label{ext-hom-comparison}
 \theta^i_{\A,\D}=\theta^i_{\A,\D}(X,Y)\:
 \Ext_\A^i(X,Y)\rarrow\Hom_\D(X,Y[i])
 \quad\text{for all $X$, $Y\in\A$, $\,i\ge0$.}
\end{equation}
 A t\+structure $(\D^{\le0},\.\D^{\ge0})$ is said to be \emph{of
the derived type} if the maps $\theta^i_{\A,\D}(X,Y)$ are isomorphisms
for all $X$, $Y\in\A$ and $i\ge0$ (see~\cite[Remarque 3.1.17]{BBD}, \cite[Corollary~A.17]{Partin} or \cite[Section~4]{PS} for further details).

\begin{lem} \label{derived-type-lemma}
 In the context of Lemma~\textup{\ref{lifting-t-structure-lemma}},
the t\+structure $({}'\D^{\le0},\.{}'\D^{\ge0})$ on the triangulated
category\/ ${}'\D$ is of the derived type if and only if
the t\+structure $(\D^{\le0},\.\D^{\ge0})$ on the triangulated
category\/ $\D$ is.
\end{lem}

\begin{proof}
 According to Lemma~\ref{lifting-t-structure-lemma}, the functor
$F\:{}'\!\.\A\rarrow\A$ is an equivalence of categories.
 So is, according to the proof of Lemma~\ref{lifting-t-structure-lemma},
the functor $F\:{}'\D^{\ge0}\rarrow\D^{\ge0}$.
 It remains to observe that the domain of
the map~\eqref{ext-hom-comparison} is an Ext group computed in
the abelian heart of the t\+structure, while the codomain is
a Hom group in the coaisle: $\Hom_\D(X,Y[i])=\Hom_\D(X[-i],Y)$,
and both the objects $X[-i]$ and $Y$ belong to~$\D^{\ge0}$.
\end{proof}

 The following lemma describes the situation in which we want to apply
Lemma~\ref{lifting-t-structure-lemma}.

\begin{lem} \label{coresolving-subcategory-lemma}
 Let\/ $\A$ be an abelian category and\/ $\E\subset\A$ be a coresolving
subcategory, viewed as an exact category with the exact category
structure inherited from\/~$\A$.
 Then the functor between the derived categories of bounded below
complexes\/ $\D^+(\E)\rarrow\D^+(\A)$ induced by the exact embedding
functor\/ $\E\rarrow\A$ is a triangulated equivalence.
 The inverse functor to this equivalence $\D(\A)\supset\D^+(\A)\rarrow \D^+(\E)\subset\D(\E)$ is a partially defined right adjoint functor (in a sense analogous to the statement of Lemma~\ref{lifting-t-structure-lemma}) to
the functor between the unbounded derived categories\/ $\D(\E)\rarrow \D(\A)$ induced by the exact embedding\/ $\E\rarrow\A$.
\end{lem}

\begin{proof}
 For any bounded below complex $A^\bu$ in $\A$ there exists
a bounded below complex $E^\bu$ in $\E$ together with
a quasi-isomorphism $A^\bu\rarrow E^\bu$ of complexes in $\A$
\cite[Lemma~I.4.6(1)]{Hart66}.
 Thus, the functor $\D^+(\E)\rarrow\D^+(\A)$ is essentially surjective.

 Since $\E$ is closed under the cokernels of monomorphisms, any
bounded below complex in $\E$ that is acyclic in $\A$ is
also acyclic in~$\E$.
 From this we will deduce that for any complex $E^\bu$ in $\E$ and
any bounded below complex $F^\bu$ in $\E$ the natural map
\begin{equation} \label{hom-E-hom-A}
 \Hom_{\D(\E)}(E^\bu,F^\bu)\lrarrow\Hom_{\D(\A)}(E^\bu,F^\bu)
\end{equation}
is an isomorphism, which implies both that the functor
$\D^+(\E)\rarrow\D^+(\A)$ is fully faithful (hence, a triangulated
equivalence) and that the inverse functor to it is partially right adjoint to
the canonical functor $\D(\E)\rarrow\D(\A)$.

 Indeed, an arbitrary morphism $E^\bu\rarrow F^\bu$ in the derived
category $\D(\A)$ can be represented by a fraction of morphisms of complexes
$E^\bu\rarrow X^\bu\larrow F^\bu$, where $X^\bu$ is a complex in $\A$
and $F^\bu\rarrow X^\bu$ is a quasi-isomorphism of complexes in~$\A$.
 Now the complex $X^\bu$ is acyclic in low cohomological degrees, so
for $n\ll0$ the natural morphism from $X^\bu$ to its canonical truncation
$X^\bu\rarrow \tau_{\ge n}X^\bu$ is a quasi-isomorphism of complexes
in~$\A$.
 The complex $\tau_{\ge n}X^\bu$ is bounded below, so there exists
a bounded below complex $G^\bu$ in $\E$ together with a quasi-isomorphism
$\tau_{\ge n}X^\bu\rarrow G^\bu$ of complexes in~$\A$.
 Then the composition $F^\bu\rarrow X^\bu\rarrow \tau_{\ge n}X^\bu
\rarrow G^\bu$ is a quasi-isomorphism of complexes in the exact
category~$\E$.
This allows to represent our morphism $E^\bu\rarrow F^\bu$ in
$\D(\A)$ by a fraction $E^\bu\rarrow G^\bu\larrow F^\bu$ of morphisms
of complexes in~$\E$.
 This proves surjectivity of the map~\eqref{hom-E-hom-A}.
 
 The injectivity is similar.
 If a fraction $E^\bu\rarrow X^\bu\larrow F^\bu$ vanishes in the group
$\Hom_{\D(\A)}(E^\bu,F^\bu)$, then there exists a quasi-isomorphism
$X^\bu\rarrow G^\bu$ of complexes in $\A$ such that
$E^\bu\rarrow X^\bu\rarrow G^\bu$ is null-homotopic.
 As above, we can choose $X^\bu\rarrow G^\bu$ so that $G^\bu$ is
a bounded below complex in $\E$, and it follows that the fraction
vanishes in $\Hom_{\D(\E)}(E^\bu,F^\bu)$ as well.
\end{proof}

 Given an abelian category $\A$ with a coresolving subcategory $\E\subset\A$,
for any complex $E^\bu$ in $\E$ we denote by $H^n_\A(E^\bu)\in\A$
the cohomology objects of the complex $E^\bu$ viewed as a complex in~$\A$.
 Consider the following two full subcategories in the unbounded derived
category $\D(\E)$:
\begin{itemize}
\item $\D_\A^{\le0}(\E)\subset\D(\E)$ is the full subcategory of all complexes
$E^\bu$ in $\E$ such that $H^n_\A(E^\bu)=0$ for all $n>0$;
\item $\D^{\ge0}(\E)\subset\D(\E)$ is the full subcategory of all objects in
$\D(\E)$ that can be represented by complexes $E^\bu$ in $\E$ with
$E^n=0$ for all $n<0$.
\end{itemize}
 As in the usual notation, for any $n\in\boZ$ we set
$\D_\A^{\le n}(\E) =\D_\A^{\le0}(\E)[-n]\subset\D(\E)$ and
$\D^{\ge n}(\E)=\D^{\ge0}(\E)[-n]\subset\D(\E)$.

\begin{prop} \label{coresolving-t-structure}
 Let\/ $\A$ be an abelian category and\/ $\E\subset\A$ be a coresolving
subcategory.
 Then the pair of full subcategories $(\D_\A^{\le0}(\E),\.\D^{\ge0}(\E))$
is a t\+structure on the unbounded derived category\/ $\D(\E)$ of
the exact category\/~$\E$.
 Moreover, this is a t\+structure of the derived type, and
the triangulated functor\/ $\D(\E)\rarrow\D(\A)$ induced by
the exact embedding\/ $\E\rarrow\A$ identifies its heart\/
$\D_\A^{\le0}(\E)\cap\D^{\ge0}(\E)$ with the abelian category\/~$\A$.
\end{prop}

\begin{proof}
% All the assertions follow from
%Lemmas~\ref{lifting-t-structure-lemma}\+-%
%\ref{coresolving-subcategory-lemma}.
% For clarity, let us summarize the key points of the argument.
%
 We apply Lemma~\ref{lifting-t-structure-lemma}
to the situation described in Lemma~\ref{coresolving-subcategory-lemma},
where ${}'\D=\D(\E)$, $\D=\D(\A)$, and $F\:\D(\E)\rarrow\D(\A)$ is
the canonical functor.
 Moreover, we set $(\D^{\le0},\.\D^{\ge0})$ to be the canonical
t\+structure on $\D(\A)$, which is certainly of the derived type.
Then $G = F|_{\D^{\ge0}(\E)}^{-1}\:\D^{\ge0}\rarrow\D^{\ge0}(\E)\subset\D(\E)$
is a partially defined right adjoint to $F$ in the sense of 
Lemma~\ref{lifting-t-structure-lemma} and
$(\D_\A^{\le0}(\E),\.\D^{\ge0}(\E))$ is precisely
the lifted t\+structure from the conclusion of the lemma.
It is of the derived type by Lemma~\ref{derived-type-lemma}.

 For clarity, we summarize the construction of
the t\+structure truncations $\tau^\E_{\le0}E^\bu$ and
$\tau^\E_{\ge1}E^\bu$ for a given complex $E^\bu$ over $\E$.
% in order to construct its canonical
%truncations $\tau^\E_{\le0}E^\bu$ and $\tau^\E_{\ge1}E^\bu$
%in the t\+structure $(\D_\A^{\le0}(\E),\.\D^{\ge0}(\E))$ on $\D(\E)$,
One first considers its canonical truncation $\tau^\A_{\ge1}E^\bu$
as a complex in $\A$, in the standard t\+structure on~$\D(\A)$.
 So $\tau^\A_{\ge1}E^\bu$ is a complex in $\A$ with the terms
concentrated in the cohomological degrees~$\ge\nobreak1$; hence there
exists a complex $F^\bu$ in $\E$ with the terms concentrated in
the cohomological degrees~$\ge\nobreak1$ endowed with a quasi-isomorphism
$\tau^\A_{\ge1}E^\bu\rarrow F^\bu$ of complexes in~$\A$.
 One sets $\tau^\E_{\ge1}E^\bu=F^\bu$, and $\tau^\E_{\le0}E^\bu$
is a cocone of the morphism of complexes $E^\bu\rarrow F^\bu$
in~$\D(\E)$.
\end{proof}

\begin{rem}
 It is instructive to look into (non)degeneracy properties of
the t\+structure $(\D_\A^{\le0}(\E),\.\D^{\ge0}(\E))$ on $\D(\E)$.
 The intersection $\bigcap_{n\ge0}\D^{\ge n}(\E)\subset\D(\E)$ consists
of some bounded below complexes in $\E$ with vanishing cohomology
in~$\A$.
 All such complexes are acyclic in $\E$, so this intersection is
a zero category.
 On the other hand, the intersection $\bigcap_{n\le 0}\D^{\le n}_\A(\E) \subset\D(\E)$ consists of all the complexes in $\E$ with vanishing
cohomology in~$\A$.
 This is precisely the kernel of the triangulated functor $\D(\E)\rarrow
\D(\A)$, and it can very well be nontrivial. Indeed, let $k$ be a field,
$\A=k[x]/(x^2)\modl$ and $\E=\A_\inj$ (see also Example~\ref{locally-noetherian-ex1}
below). Since the complex
\[ \cdots \lrarrow k[x]/(x^2) \overset{x}\lrarrow k[x]/(x^2) \overset{x}\lrarrow k[x]/(x^2) \lrarrow \cdots \]
is acyclic but not contractible, it is non-zero in $\D(\E)=\Hot(\A_\inj)$,
but it becomes zero in $\D(\A)$.
\end{rem}

 Let us formulate the dual assertions.
 Given an abelian category $\B$ with a resolving subcategory
$\F\subset\B$, for any complex $F^\bu$ in $\F$ we denote by
$H^n_\B(F^\bu)\in\B$ the cohomology objects of the complex $F^\bu$
viewed as a complex in~$\B$.
 Consider the following two subcategories in the unbounded derived
category~$\D(\F)$:
\begin{itemize}
\item $\D^{\le0}(\F)\subset\D(\F)$ is the full subcategory of all objects
in $\D(\F)$ that can be represented by complexes $F^\bu$ in $\F$ with
$F^n=0$ for all $n>0$;
\item $\D_\B^{\ge0}(\F)\subset\D(\F)$ is the full subcategory of all
complexes $F^\bu$ in $\F$ such that $H^n_\B(F^\bu)=0$ for all $n<0$.
\end{itemize}

\begin{prop}
 Let\/ $\B$ be an abelian category and\/ $\F\subset\B$ be a resolving
subcategory.
 Then the pair of full subcategories $(\D^{\le0}(\F),\.\D_\B^{\ge0}(\F))$
is a t\+structure on the unbounded derived category\/ $\D(\F)$ of
the exact category\/~$\F$.
 Moreover, this is a t\+structure of the derived type, and the triangulated
functor\/ $\D(\F)\rarrow\D(\B)$ induced by the exact embedding\/
$\F\rarrow\B$ identifies its heart $\D^{\le0}(\F)\cap\D_\B^{\ge0}(\F)$
with the abelian category\/~$\B$.
\end{prop}

\begin{proof}
 Dual to Proposition~\ref{coresolving-t-structure}.
\end{proof}

 Now we are well-equipped for the discussion of $\infty$\+tilting and
$\infty$\+cotilting t\+structures.
 Let $\A$ be a complete, cocomplete abelian category with an injective
cogenerator $J$ and an $\infty$\+tilting pair $(T,\.\E)$, and let $\B$
be the corresponding complete, cocomplete abelian category with
a projective generator $P$ and an $\infty$\+cotilting pair $(W,\.\F)$,
as in Corollary~\ref{tilting-cotilting-correspondence-cor}.
Suppose further for convenience that $\E$, and hence also $\F$, are idempotent complete.
 Then the exact category $\E\simeq\F$ is simultaneously a coresolving
subcategory in $\A$ and a resolving subcategory in~$\B$.
{\hbadness=1100\par}

 Thus we have two t\+structures $(\D_\A^{\le0}(\E),\.\D^{\ge0}(\E))$
and $(\D^{\le0}(\F),\.\D_\B^{\ge0}(\F))$ on the unbounded derived
category $\D(\E)=\D=\D(\F)$.
 The hearts of these t\+structures are the abelian categories $\A$
and $\B$, respectively.

 Looking from the point of view of the category $\A$, the t\+structure
$(\D_\A^{\le0}(\E),\.\D^{\ge0}(\E))$ on the triangulated category $\D$
can be called the \emph{standard t\+structure}, and the t\+structure 
$(\D^{\le0}(\F),\.\D_\B^{\ge0}(\F))$ is
the \emph{$\infty$\+tilting t\+structure}.
 Looking from the point of view of the category $\B$, the t\+structure
$(\D^{\le0}(\F),\.\D_\B^{\ge0}(\F))$ on the triangulated category $\D$
is the \emph{standard t\+structure}, and the t\+structure
$(\D_\A^{\le0}(\E),\.\D^{\ge0}(\E))$ is
the \emph{$\infty$\+cotilting t\+structure}.
 The abelian category $\B$ is the \emph{$\infty$\+tilting heart}, and
the abelian category $\A$ is the \emph{$\infty$\+cotilting heart}.

\Section{Examples}
\label{examples-secn}

\begin{ex} \label{n-co-tilting-ex}
 Let $\A$ be a complete, cocomplete abelian category with an injective
cogenerator $J$ and an $\infty$\+tilting object $T\in\A$, and let $\B$
be the corresponding complete, cocomplete abelian category with
a projective generator $P$ and an $\infty$\+cotilting object $W\in\B$,
as in Corollary~\ref{tilting-cotilting-cor}.
 In this context, if \emph{both} the projective dimension of
the $\infty$\+tilting object $T\in\A$ and the injective dimension of
the $\infty$\+cotilting object $W\in\B$ are finite, then
they are equal to each other, $\pd_\A T=n=\id_\B W$.
 Furthermore, this holds if and only if the object $T\in\A$ is $n$\+tilting
if and only if the object $W\in\B$ is $n$\+cotilting (both in the sense
of~\cite[Sections~2 and~4]{PS}).

 Indeed, suppose that $\pd_\A T<\infty$ and $\id_\B W<\infty$ and
denote by $n$ the maximum of the two values.
 Then the left exact functor $\Psi\:\A\rarrow\B$ has finite homological
dimension, since it can be computed as the functor $\Hom_\A(T,{-})$;
and the right exact functor $\Phi\:\B\rarrow\A$ has finite homological
dimension, since it can be computed as the functor
$\Hom_\B({-},W)^\sop$.
 Denote by $\E_T\subset\A$ the full subcategory of all objects $E\in\A$
such that $\Ext_\A^i(T,E)=0$ for all $i>0$, and by $\F_W\subset\B$
the full subcategory of all objects $F\in\B$ such that
$\Ext_\B^i(F,W)=0$ for all $i>0$.
 (By the definition, we have $\E_\mx(T)\subset\E_T$ and
$\F_\mx(W)\subset\F_W$.)

 Then the functor $\Psi$ is exact on the exact category $\E_T$ and
the functor $\Phi$ is exact on the exact category~$\F_W$.
 The full subcategory $\E_T$ is coresolving in $\A$, and the full
subcategory $\F_W$ is resolving in $\B$, with both the (co)resolution
dimensions bounded by the finite constant $n$. The latter fact is due to
the observation that, thanks to a simple dimension shifting argument,
any $n$-th cosyzygy object in $\A$ belongs to $\E_T$
and any $n$-th syzygy object in $\B$ belongs to $\F_W$.

 Let us show that the functors $\Phi$ and $\Psi$ restrict to mutually
inverse equivalences between the exact categories $\E_T$ and $\F_W$.
 Given an object $E\in\E_T$, choose an exact sequence $0\rarrow E
\rarrow J^0\rarrow\dotsb\rarrow J^{d-1}\rarrow E'\rarrow0$ in $\A$
with $J^i\in\A_\inj$ with $d\ge\operatorname{max}(n,2)$.
 Then the sequence $0\rarrow\Psi(E)\rarrow\Psi(J^0)\rarrow\dotsb
\rarrow\Psi(J^{d-1})\rarrow\Psi(E')\rarrow0$ is exact in~$\B$, and
the objects $\Psi(J^i)$ belong to the full subcategory
$\Prod(W)\subset\F_W\subset\B$.
 Hence $\Psi(E)\in\F_W$ by dimension shifting.

 Furthermore, we have $E'\in\E_T$, hence $\Psi(E')\in\F_W$.
 It follows that the sequence $0\rarrow\Phi\Psi(E)\rarrow\Phi\Psi(J^0)
\rarrow\dotsb\rarrow\Phi\Psi(J^{d-1})\rarrow\Phi\Psi(E')\rarrow0$ is
exact in~$\A$.
 Since the adjunction morphisms $\Phi\Psi(J^i)\rarrow J^i$ are
isomorphisms for $i=0$ and~$1$, so is the adjunction morphism
$\Phi\Psi(E)\rarrow E$.
 Similarly one shows that $\Phi(F)\in\E_T$ for all $F\in\F_W$, and
the adjunction morphism $F\rarrow\Psi\Phi(F)$ is an isomorphism.

 According to~\cite[Lemmas 5.4.1 and 5.4.2]{BvdB},
\cite[Proposition~1.5]{FMS} or~\cite[Theorem~5.5]{PS}
and the references therein, the triangulated functors $\D(\E_T)\rarrow
\D(\A)$ and $\D(\F_W)\rarrow\D(\B)$ induced by the exact embedding
functors $\E_T\rarrow\A$ and $\F_W\rarrow\B$ are equivalences of
triangulated categories.
 Thus we obtain a triangulated equivalence
$$
 \D(\A)\simeq\D(\E_T)=\D(\F_W)\simeq\D(\B).
$$
 Applying, e.~g., \cite[Proposition~2.5 and Corollary~4.4(b)]{PS},
one can conclude that the conditions~(i\+iii) and~(i*-iii*)
of~\cite[Sections~2 and~4]{PS} hold for $T$ and~$W$, respectively.
That is, $T$ is $n$-tilting, $W$ is $n$-cotilting and, moreover,
$\pd_\A T=n=\id_\B W$ by~\cite[Corollary~4.12]{PS}.

 Following~\cite[Lemma~5.1]{PS}, the two conditions~(i$_\mx$)
and~(ii$_\mx$) defining the full subcategory $\E_\mx(T)\subset\A$
are equivalent in this case.
 Similarly, the two conditions~(i$_\mx^{\textstyle*}$)
and~(ii$_\mx^{\textstyle*}$) defining the full subcategory
$\F_\mx(W)\subset\B$ are equivalent.
 So either one of the two conditions is sufficient to define
these classes in the $n$\+(co)tilting case, and we actually have
$\E_\mx(T)=\E_T$ and $\F_\mx(W)=\F_W$.
 It is only in the $\infty$\+(co)tilting situation that we need to
impose both the conditions.
 The full subcategory $\E=\E_\mx(T)$ is the \emph{$n$\+tilting class}
of an $n$\+tilting object $T$, and the full subcategory $\F=\F_\mx(W)$
is the \emph{$n$\+cotilting class} of an $n$\+cotilting object~$W$,
as discussed in~\cite[Sections~3\+-4]{PS}.
 According to~\cite[Lemma~5.3 and Remark~5.4]{PS}, both the full
subcategories $\E$ and $\F$ are closed under both the infinite
products and coproducts in $\A$ and~$\B$.

 Finally, note that if $T$ is $n$\+tilting, then $W$ is $n$\+cotilting 
and vice versa by~\cite[Corollary~4.12]{PS}.
Thus, both the projective dimension of $T$ and
the injective dimension of $W$ need to be finite for either
of the two objects to be $n$\+(co)tilting.
\end{ex}

\begin{ex}
 Let $\A$ be a complete, cocomplete abelian category with an injective
cogenerator $J$ and an $\infty$\+tilting pair $(T,\.\E)$, and let $\B$ be
the corresponding complete, cocomplete abelian category with
a projective generator $P$ and an $\infty$\+cotilting pair $(W,\.\F)$, as
in Corollary~\ref{tilting-cotilting-correspondence-cor}.
 Suppose that the full subcategory $\E\subset\A$ is closed under
coproducts and the full subcategory $\F\subset\B$ is closed
under products.
 Then, by Proposition~\ref{co-products-closed-exact-subcategory},
the triangulated functors $\D(\E)\rarrow\D(\A)$ and $\D(\F)\rarrow
\D(\B)$ induced by the exact embeddings $\E\rarrow\A$ and
$\F\rarrow\B$ are Verdier quotient functors.

 Assume that only \emph{one} of the objects $T$ and $W$ has finite
homological dimension, or more specifically, that $\pd_\A T<\infty$.
 Then the left exact functor $\Psi\:\A\rarrow\B$ has finite homological
dimension and the full subcategory
$\E_T = \{E\in\A\mid\Ext_\A^i(T,E)=0\ \forall\, i>0\}$
of $\A$ has finite coresolution dimension, as in the previous example.
 In particular, the complex $\Psi(E^\bu)$ is acyclic in $\B$ for any
complex $E^\bu$ in the category $\E$ that is acyclic in~$\A$.
 So the composition of triangulated functors $\D(\E)\simeq\D(\F)
\rarrow\D(\B)$ factorizes through the Verdier quotient functor
$\D(\E)\rarrow\D(\A)$, or in other words, the triangulated equivalence
$\D(\E)\simeq\D(\F)$ descends to a triangulated functor
$\D(\A)\rarrow\D(\B)$ in Figure~\ref{equivalences-coderived-fig}.
 This is also a Verdier quotient functor (since such is the functor
$\D(\F)\rarrow\D(\B)$).

 Similarly, assume that $\id_\B W<\infty$.
 Then the right exact functor $\Phi\:\B\rarrow\A$ has finite
homological dimension.
 In particular, the complex $\Phi(F^\bu)$ is acyclic in $\A$ for any
complex $F^\bu$ in the category $\F$ that is acyclic in~$\B$.
 Hence the triangulated equivalence $\D(\F)\simeq\D(\E)$ descends
to a triangulated Verdier quotient functor $\D(\B)\rarrow\D(\A)$.

In the representation theory of finite-dimensional algebras, it is
an open problem whether a finite-dimensional $\infty$\+tilting module
of finite projective dimension is already $n$\+tilting for some~$n$.
It goes under the name of the Wakamatsu tilting conjecture, and
it is a member of a family of long standing so-called
homological conjectures for finite-dimensional algebras
\cite[Section~4]{MR}, \cite[\S IV.3]{BR}.
\end{ex}

\begin{ex} \label{locally-noetherian-ex1}
 Let $\A$ be a locally Noetherian Grothendieck abelian category
(cf.~\cite[Section~10.2]{PS}).
 Choose an injective object $J\in\A$ such that $\A_\inj=\Add(J)$;
then it follows that $J$ is an injective cogenerator of $\A$, and
one also has $\A_\inj=\Prod(J)$.
 Set $T=J$ and $\E=\A_\inj\subset\A$.
 Then $(T,\.\E)$ is an $\infty$\+tilting pair in~$\A$.

 In the corresponding abelian category $\B$ with a natural
projective generator $P$ \cite[Theorem~3.6]{Prev}, one has
$\B_\proj=\Add(P)=\Prod(P)$ (see also Lemma~\ref{co-products-preservation}).
 The related $\infty$\+cotilting pair in $\B$ is $(W,\.\F)$, where
$W=P$ and $\F=\B_\proj$.
 So both the full subcategories $\E\subset\A$ and $\F\subset\B$
are closed under both the products and coproducts.
 As always in the context of
Corollary~\ref{tilting-cotilting-correspondence-cor},
one has an equivalence of additive/exact categories $\E\simeq\F$.

 The derived category $\D(\E)$ is simply the homotopy category
$\Hot(\A_\inj)$; it is equivalent to the coderived category
$\D^\co(\A)$ (see the argument for~\cite[Theorem~2.4]{PcohSGdual}).
 The derived category $\D(\F)$ is simply the homotopy category
$\Hot(\B_\proj)$; it is equivalent to the contraderived category
$\D^\ctr(\B)$ (cf.~\cite[Theorem~4.4(b)]{PcohSGdual}
and~\cite[Corollary~A.6.2]{Pcosh}).
 Hence the derived equivalence
$$
 \D^\co(\A)\simeq\Hot(\A_\inj)=\Hot(\B_\proj)\simeq\D^\ctr(\B).
$$

 These are the \emph{minimal} $\infty$\+tilting and
$\infty$\+cotilting pair for the $\infty$\+tilting object $T\in\A$
and the $\infty$\+cotilting object $W\in\B$, in the sense of
Example~\ref{minimal-pairs-ex}: one has $\E_\mn(T)=\E=\A_\inj$
and $\F_\mn(W)=\F=\B_\proj$.
\end{ex}

\begin{ex} \label{locally-noetherian-ex2}
 In the context of the previous example, it is also instructive to
consider the \emph{maximal} $\infty$\+tilting pair $(T,\.\E_\mx(T))$
for the $\infty$\+tilting object $T=J$ in the category $\A$ and
the maximal $\infty$\+cotilting pair $(W,\.\F_\mx(W))$ for
the $\infty$\+cotilting object $W=P$ in the category~$\B$.

 The full subcategory\ $\E_\mx(T)\subset\A$ consists of all
the objects $E\in\A$ for which there exists an unbounded acyclic
complex of injective objects
$$
 \dotsb\lrarrow J^{-2}\lrarrow J^{-1}\lrarrow J^0\lrarrow J^1
 \lrarrow J^2\lrarrow\dotsb
$$
such that the complex $\Hom_\A(J,J^\bu)$ is acyclic and $E$ is
the image of the morphism $J^{-1}\rarrow J^0$.
 This is known as the full subcategory of \emph{Gorenstein injective
objects} in the abelian category~$\A$.

 Similarly, the full subcategory $\F_\mx(W)\subset\B$ consists of all
the objects $F\in\B$ for which there exists an unbounded acyclic
complex of projective objects
$$
 \dotsb\lrarrow P_2\lrarrow P_1\lrarrow P_0\lrarrow P_{-1}
 \lrarrow P_{-2}\lrarrow\dotsb
$$
such that the complex $\Hom_\B(P_\bu,P)$ is acyclic and $F$ is
the image of the morphism $P_0\rarrow P_{-1}$.
 This is known as the full subcategory of \emph{Gorenstein projective
objects} in the abelian category~$\B$ (cf.~\cite[Definition~3.7]{En}).

 Hence we can conclude from
Theorems~\ref{tilting-cotilting}\+-\ref{cotilting-tilting}
that the exact categories of Gorenstein injective objects in $\A$
and Gorenstein projective objects in $\B$ are naturally equivalent.

 If $\A$ has a generating set of objects of finite projective dimension,
%(e.g.\ if $\A$ is the category of quasi-coherent sheaves on
%a quasi-projective scheme over a commutative Noetherian ring),
then, by~\cite[Theorem 5.7]{Gil} or~\cite[Lemma 7.2]{St-coder},
the class of acyclic complexes of injectives is closed under products
(although products may not be exact in~$\A$).
In particular, $\E_\mx(T)\subset\A$ is closed under products, and so
is $\F_\mx(W)\subset\B$ by Lemma~\ref{co-products-preservation}.
 Dually, if $\B$ has a cogenerating set of objects of finite injective
dimension, then both $\E_\mx(T)\subset\A$ and $\F_\mx(W)\subset\B$
are closed under coproducts.

 In particular, if $\A$ is the category of quasi-coherent sheaves on
a quasi-compact semi-separated scheme $X$, then any quasi-coherent
sheaf on $X$ is the quotient of one of the so-called very flat
quasi-coherent sheaves~\cite[Lemma~4.1.1]{Pcosh}
(see~\cite[Section~2.4]{M-n} or~\cite[Lemma~A.1]{EP} for the more
widely known, but weaker assertion with flat sheaves in place of
the very flat ones).
 If $X$ is covered by $n$~affine open subschemes, then the projective
dimension of any very flat quasi-coherent sheaf, as an object of $\A$,
does not exceed~$n$, as one can show using a \v Cech resolution for
the affine covering, together with the fact that the projective
dimension of a very flat module does not exceed~$1$
(cf.~\cite[properties~(VF5) and~(VF6)]{PSl}).
 Thus the class of acyclic complexes of injectives is closed under
products in~$\A$.
 If $X$ is also Noetherian, then $\A$ is a locally Noetherian category,
and the discussion in the previous paragraph applies.
\end{ex}

\begin{ex} \label{locally-noetherian-gorenstein}
 In the context of Examples~\ref{locally-noetherian-ex1}\+-%
\ref{locally-noetherian-ex2}, one can say that a locally Noetherian
Grothendieck abelian category $\A$ is \emph{$n$\+Gorenstein} if
the $\infty$\+tilting object $T=J$ is $n$\+tilting.
 This means that $\pd_\A T=\id_\B W\le n$ (cf.~\cite[Theorem~10.3]{PS}).

 In this case, we have the minimal $\infty$\+tilting and
$\infty$\+cotilting pair $(T,\.\E_\mn(T))$ and $(W,\.\F_\mn(W))$
with $\E_\mn(T)=\A_\inj$ and $\F_\mn(W)=\B_\proj$, as in
Example~\ref{locally-noetherian-ex1}.
 We also have the maximal $\infty$\+tilting and
$\infty$\+cotilting pair $(T,\.\E_\mx(T))$ and $(W,\.\F_\mx(W))$
with $\E_\mx(T)=\E_T$ being the $n$\+tilting class of the $n$\+tilting
object $T\in\A$ (consisting of all the Gorenstein injectives
in~$\A$) and $\F_\mx(W)=\F_W$ being the $n$\+cotilting class of
the $n$\+cotilting object $W\in\B$ (consisting of all
the Gorenstein projectives in~$\B$), as in
Examples~\ref{n-co-tilting-ex} and~\ref{locally-noetherian-ex2}.

 The two related derived equivalences (as in
Section~\ref{derived-equivalence-secn}) form a commutative diagram
with the natural Verdier quotient functors
$$
\begin{diagram}
\node{\D^\co(\A)}\arrow{s,A}\arrow[2]{e,=}
\node[2]{\Hot(\A_\inj)}\arrow{s,A}\arrow[2]{e,=}
\node[2]{\Hot(\B_\proj)}\arrow{s,A}\arrow[2]{e,=}
\node[2]{\D^\ctr(\B)}\arrow{s,A} \\
\node{\D(\A)}\arrow[2]{e,=}\node[2]{\D(\E_T)}\arrow[2]{e,=}
\node[2]{\D(\F_W)}\arrow[2]{e,=}\node[2]{\D(\B)}
\end{diagram}
$$
\end{ex}

\begin{ex}
 Let $A$ and $B$ be associative rings, and let $C$ be
an $A$\+$B$\+bimodule.
 One says that $C$ is a \emph{semidualizing bimodule} (in
the terminology of~\cite{HW}) or a \emph{pseudo-dualizing
bimodule} (in the terminology of~\cite{Pps}, which we adopt here) for
the rings $A$ and $B$ if the following conditions are satisfied:
\begin{itemize}
\item the left $A$\+module $C$ has a projective resolution by
finitely generated projective left $A$\+modules, and the right
$B$\+module $C$ has a projective resolution by finitely generated
projective right $B$\+modules;
\item the homothety maps $A\rarrow\Ext_{B^\rop}^*(C,C)$ and
$B^\rop\rarrow\Ext_A^*(C,C)$ are isomorphisms of graded rings
(where $B^\rop$ denotes the opposite ring to~$B$).
\end{itemize}
This definition is (essentially) obtained by dropping the finite
injective dimension condition in the definition of a dualizing
module over a pair of associative rings.

 Let $C$ be a pseudo-dualizing $A$\+$B$\+bimodule.
 Set $\A=A\modl$ and $\B=B\modl$ to be the abelian categories of 
left modules over the rings $A$ and~$B$.
 Then $T=C$ is a (finitely generated) $\infty$\+tilting object
in~$\A$.
 The related maximal $\infty$\+tilting class $\E_\mx(T)\subset\A$
is known as the \emph{Bass class} \cite[Theorem~6.1]{HW}, and
it contains the injective left $A$\+modules by~\cite[Lemma~4.1]{HW}.

 The corresponding tilted abelian category is $\sigma_T(\A)=\B$, and
its natural projective generator is $P=B$.
 Choosing $J=\Hom_\boZ(A,\boQ/\boZ)$ as the injective cogenerator
of $\A$, the corresponding $\infty$\+cotilting object in $\B$
is $W=\Hom_\boZ(C,\boQ/\boZ)$.
 The related maximal $\infty$\+cotilting class $\F_\mx(W)\subset\B$
is known as the \emph{Auslander class} \cite[Theorem~2]{HW}.
 The objects of the full subcategory $\Add(T)\subset\A$ are called
\emph{$C$\+projectives} in~\cite{HW}, and the objects of the full
subcategory $\Prod(W)\subset\B$ are called \emph{$C$\+injectives}.
 The equivalence of exact categories $\E_\mx(T)\simeq\F_\mx(W)$
is a part of what is known as the \emph{Foxby
equivalence}~\cite[Theorem~1 or Proposition~4.1]{HW}.

 Both the full subcategories $\E_\mx(T)\subset\A$ and $\F_\mx(W)
\subset\B$ are closed under both the infinite products and
coproducts~\cite[Proposition~4.2]{HW}, so the results of our
Section~\ref{derived-equivalence-secn} apply and provide a commutative
diagram of a triangulated equivalence and Verdier quotient functors
$$
\begin{diagram}
\node{\D^\co(A\modl)}\arrow{s,A}\node[2]{\D^\ctr(B\modl)}\arrow{s,A} \\
\node{\D(\E_\mx(T))}\arrow{s,A}\arrow[2]{e,=}\node[2]{\D(\F_\mx(W))}
\arrow{s,A} \\ \node{\D(A\modl)}\node[2]{\D(B\modl)}
\end{diagram}
$$

 The paper~\cite{Pps} is devoted to generalizing this theory to
the case of a pseudo-dualizing \emph{complex} of bimodules.
 In particular, (a coproduct and product-closed version of)
the \emph{minimal} $\infty$\+tilting and $\infty$\+cotilting
classes for $T$ and $W$ is discussed in~\cite[Section~5]{Pps}.
\end{ex}

\begin{ex} \label{coring-example}
 Let $\cC$ be a coassociative, counital coring over an associative
ring~$A$ (see~\cite[Section~10.3]{PS}).
 Assume that $\cC$ is a projective left and a flat right $A$\+module.
 Let $\A=\cC\comodl$ be the category of left $\cC$\+comodules; it is
a Grothendieck abelian category.
 Set $T\in\A$ to be the cofree left $\cC$\+comodule $T=\cC$.
 We claim that $T$ is an $\infty$\+tilting object in~$\A$.

 Indeed, it was explained in~\cite[Section~10.3]{PS} that $T$ is
weakly tilting, so it remains to show that the injective objects of $\A$
satisfy the condition~(ii$_\mx$).
 A left $\cC$\+comodule is injective if and only if it is a direct
summand of a $\cC$\+comodule $\cC\ot_A I$ coinduced from an injective
left $A$\+module~$I$ \cite[Sections~1.1.2 and~5.1.5]{Psemi}.
 Now applying the coinduction functor $\cC\ot_A{-}$ to a projective
resolution of the $A$\+module $I$ produces an $\Add(T)$\+resolution
of the $\cC$\+comodule $\cC\ot_A I$ as in~(ii$_\mx$).
 This resolution remains exact after applying the functor
$\Hom_\A(T,{-})$, because $\Hom_\cC(\cC,\>\cC\ot_AV)\simeq
\Hom_A(\cC,V)$ and $\cC$ is a projective left $A$\+module.

 The abelian category $\B=\sigma_T(\A)$ is the category of
left $\cC$\+contramodules, $\B=\cC\contra$ \cite[Section~10.3]{PS}.
 The natural projective generator is $P=\Hom_\cC(\cC,\cC)=
\Hom_A(\cC,A)$.
 Given an injective cogenerator $I$ of the category of left
$A$\+modules, one can choose $J=\cC\ot_A I$ as the injective
cogenerator of $\A=\cC\comodl$; then the related cotilting object
in $\B=\cC\contra$ is $W=\Hom_\cC(\cC,\>\cC\ot_AI)=\Hom_A(\cC,I)$.

 When the left homological dimension of the ring $A$ is finite,
we can describe the minimal class $\E_\mn(T)$ which forms
an $\infty$\+tilting pair with $T$ (Example~\ref{minimal-pairs-ex})
more explicitly as the  full subcategory $\E\subset\A$ of all
$\cC/A$\+injective left $\cC$\+comodules~\cite[Sections~5.1.4
and~5.3]{Psemi}, \cite[Section~3.4]{Prev}.
 Here, a left $\cC$\+comodule $\cM$ is called \emph{$\cC/A$\+injective}
if $\Ext^i_\cC(\cL,\cM) = 0$ for all $i>0$ and all left $\cC$\+comodules
$\cL$ with projective underlying left $A$\+modules.
 The class $\E$ is coresolving and contains all the coinduced
$\cC$\+comodules $\cC\otimes_AM$, \ $M\in A\modl$.
 In particular, $\Add(T)\subset \E$, objects of $\Add(T)$ are by
definition projective in $\E$, and by~\cite[Lemma~5.2(a) and
proof of Lemma~5.3.2(a)]{Psemi}, there are enough such projectives.
 On the other hand, $\E$ has enough injectives, $\E_\inj = \A_\inj$,
and the proof of~\cite[Theorem~5.3]{Psemi} reveals that any object
of $\E$ has finite injective dimension bounded by the left homological
dimension of~$A$. Now we can use the following observation.

\begin{lem}
 Let $(T,\E)$ be an $\infty$\+tilting pair in a complete, cocomplete
abelian category\/ $\A$ with an injective cogenerator.
 If\/ $\E$ has finite homological dimension as an exact category,
then\/ $\E = \E_\mn(T)$.
\end{lem}

\begin{proof}
 If $n$~is the homological dimension of $\E$, then any object $E\in\E$
admits a long exact sequence
\[ 0 \lrarrow E \lrarrow J^0 \lrarrow \dotsb \lrarrow J^n \lrarrow 0 \]
in $\E$ with $J^0$,~\dots, $J^n\in\E_\inj=\A_\inj$.
 Since this sequence remains exact after applying $\Hom_\A(T,{-})$,
it follows that $E\in\E_\mn(T)$ by the conditions~(i) and~(v) from
Section~\ref{infty-til-cotil-pairs-secn}.
\end{proof}

 Dually, the full subcategory $\F\subset\B$ in the related $\infty$\+cotilting
pair $(W,\.\F)$ consists of all the $\cC/A$\+projective left
$\cC$\+contramodules.
 This is analogously the minimal $\infty$\+cotilting pair
for the $\infty$\+cotilting
object $W\in\B$.
%: one has $\E_\mn(T)=\E$ and $\F_\mn(W)=\F$.
 Since the class of $\cC/A$-injective comodules is closed under products and the class of $\cC/A$-projective contramodules is closed under coproducts, 
 both the full subcategories $\E\subset\A$ and $\F\subset\B$ are closed under both the infinite products and coproducts
 by Lemma~\ref{co-products-preservation}.
 The related derived equivalence is~\cite[Section~5.4]{Psemi}
$$
 \D^\co(\cC\comodl)\simeq\D(\E)=\D(\F)\simeq\D^\ctr(\cC\contra).
$$

 For comparison, when $\cC$ is a left Gorenstein coring
in the sense of~\cite[Section~10.3]{PS}, i.e.\ $T\in\A$
is an $n$\+tilting object, considering the corresponding tilting and
cotilting classes $\E_\mx(T)\subset\A$ and $\F_\mx(W)\subset\B$
produces a triangulated equivalence between the conventional derived
categories, $\D(\cC\comodl)\simeq\D(\cC\contra)$.
%, as explained in~\cite[Example~5.3]{PS}.
\end{ex}

\begin{ex}
 The case of a coassociative coalgebra $\cC$ over a field~$k$
is a common particular case of
Examples~\ref{locally-noetherian-ex1}\+-\ref{locally-noetherian-ex2}
and Example~\ref{coring-example}.
 It is also a particular case of the next
Example~\ref{semialgebra-example}.

 In this case, one has $\A=\cC\comodl$ and $\B=\cC\contra$.
 The $\infty$\+tilting object $T=\cC=J$ is the natural injective
cogenerator of the locally Noetherian Grothendieck abelian
category $\A$, and the $\infty$\+cotilting object $W=\cC^*=
\Hom_k(\cC,k)=P$ is the natural projective generator of
the abelian category~$\B$.

 When $\cC$ is a Gorenstein coalgebra, we are in the situation
of Example~\ref{locally-noetherian-gorenstein}
(see~\cite[Section~10.1]{PS}).
\end{ex}

\begin{ex} \label{semialgebra-example}
 Let $\bcS$ be a semiassociative, semiunital semialgebra over
a coassociative, counital coalgebra $\cC$ over a field~$k$
(see~\cite[Section~10.3]{PS}).
 Assume that $\bcS$ is an injective left and right $\cC$\+comodule.
 Let $\A=\bcS\simodl$ be the category of left $\bcS$\+semimodules;
it is a Grothendieck abelian category.
 Set $T\in\A$ to be the semifree left $\bcS$\+semimodule $T=\bcS$,
and take $\E\subset\A$ to be the full subcategory of all left
$\bcS$\+semimodules whose underlying left $\cC$\+comodules are
injective, $\E=\bcS\simodl_{\cC\dinj}$.
 Then $(T,\.\E)$ is an $\infty$\+tilting pair in~$\A$.

 The related abelian category $\B=\sigma_T(\A)$ is the category of
left $\bcS$\+semicontra\-modules, $\B=\bcS\sicntr$
\cite[Section~10.3]{PS}.
 The natural projective generator is $P=\Hom_\bcS(\bcS,\bcS)\in\bcS\sicntr$.
 The full subcategory $\F=\Psi(\E)\subset\B$ consists of all left
$\bcS$\+semicontramodules whose underlying left $\cC$\+contramodules
are projective, $\F=\bcS\sicntr_{\cC\dproj}$.
 The $\infty$\+cotilting object $W\in\B$ corresponding to the natural
choice of an injective cogenerator $J\in\A$ is
$W=\bcS^*=\Hom_k(\bcS,k)\in\bcS\sicntr$.
 Both the full subcategories $\E\subset\A$ and $\F\subset\B$ are
closed under both the products and coproducts.
 A detailed discussion of the equivalence of exact categories
$\bcS\simodl_{\cC\dinj}\simeq\bcS\sicntr_{\cC\dproj}$ can be found
in~\cite[Section~3.5]{Prev}.

 The derived category $\D(\E)$ of the exact category $\E$ is called
the \emph{semiderived category of left\/ $\bcS$\+semimodules} and
denoted by $\D(\bcS\simodl_{\cC\dinj})=\D^\si(\bcS\simodl)$
\cite[Section~0.3.3]{Psemi}.
 Generally speaking, it is properly intermediate between the coderived
category $\D^\co(\bcS\simodl)$ and the derived category
$\D(\bcS\simodl)$.
 Similarly, the derived category $\D(\F)$ of the exact category $\F$
is called the \emph{semiderived category of left\/
$\bcS$\+semicontramodules} and denoted by
$\D(\bcS\sicntr_{\cC\dproj})=\D^\si(\bcS\sicntr)$
\cite[Section~0.3.6]{Psemi}.
 Generally speaking, it is properly intermediate between
the contraderived category $\D^\ctr(\bcS\sicntr)$ and the derived
category $\D(\bcS\sicntr)$.

 The triangulated equivalence $\D^\si(\bcS\simodl)\simeq
\D^\si(\bcS\sicntr)$ is called the \emph{derived
semimodule-semicontramodule correspondence} \cite[Sections~0.3.7
and~6.3]{Psemi}.
 For an application to representation theory of infinite-dimensional
Lie algebras (such as the Virasoro or Kac--Moody algebras),
see~\cite[Corollary~D.3.1]{Psemi}.
\end{ex}

\end{document}